\documentclass[preprint, aos, 10pt]{imsart}
\setattribute{journal}{name}{}

\RequirePackage[colorlinks,citecolor=blue,urlcolor=blue]{hyperref}

\usepackage{amssymb}
\usepackage{amsmath}
\usepackage{amsthm}
\usepackage{srcltx}
\usepackage{mathrsfs}
\usepackage{graphicx}

\usepackage{epsfig}
\usepackage{graphicx}
\usepackage{mathrsfs}
\usepackage{color}
\usepackage{euscript}

\usepackage{times}

\input aad.tex

\def\Cov{\mathop{\rm Cov}\nolimits}
\def\cov{\mathop{\rm cov}\nolimits}
\def\var{\mathop{\rm var}\nolimits}

\def\1{\mathbf{1}}
\def\lle{\lesssim}
\def\bar{\overline}

\newtheorem{theorem}{Theorem}[section]
\newtheorem{lemma}[theorem]{Lemma}

\newtheorem{assumption}[theorem]{Assumption}

\numberwithin{equation}{section}

\begin{document}

\begin{frontmatter}

\title{Bayes procedures for adaptive inference in  inverse problems for the white noise model}

\runtitle{Bayesian adaptation in inverse problems}


\author{\fnms{B.T.} \snm{Knapik}\ead[label=e1]{b.t.knapik@vu.nl}},
\thankstext{t1}{Research supported by Netherlands Organization for Scientific Research NWO}
\thankstext{t2}{This version: \today}
\address{Department of Mathematics\\ VU University Amsterdam\\ \printead{e1}}
\author{\fnms{B.T.} \snm{Szab\'o}\corref{}\ead[label=e2]{b.szabo@tue.nl}},
\address{Department of Mathematics\\ Eindhoven University of Technology\\ \printead{e2}}
\author{A.W. van der Vaart\ead[label=e3]{avdvaart@math.leidenuniv.nl}}
\address{Mathematical Institute \\Leiden University\\ \printead{e3}}\\
\and
\author{J.H. van Zanten\ead[label=e4]{hvzanten@uva.nl}}
\address{Korteweg-de Vries institute for Mathematics\\ University of Amsterdam\\ \printead{e4}}
\affiliation{VU University Amsterdam, Eindhoven University of Technology, Leiden University and University of Amsterdam}
\runauthor{Knapik, Szab\'o, Van der Vaart and Van Zanten}

\bigskip

\begin{abstract}
We study empirical and hierarchical Bayes approaches to the problem of estimating an infinite-dimensional parameter
in mildly ill-posed inverse problems. We consider a class of prior distributions indexed by a hyperparameter
that quantifies regularity. We prove that both methods we consider succeed in automatically selecting
this parameter optimally, resulting in optimal convergence rates for truths with  Sobolev or analytic ``smoothness'',
without using knowledge about this regularity. Both methods are illustrated by simulation examples.
\end{abstract}

%
%

\end{frontmatter}

\maketitle

\section{Introduction}

In recent years, Bayesian approaches have become more and more common in dealing with
nonparametric statistical inverse problems.
Such problems arise in many fields of applied science, including
geophysics, genomics, medical image analysis and astronomy, to mention but a few.
In nonparametric inverse problems some form of regularization is usually needed in order
to estimate the (typically functional) parameter of interest.
One possible explanation of the increasing popularity of Bayesian methods is the
fact that assigning a prior distribution to an unknown  functional parameter
is a natural way of specifying a degree of regularization.
Probably at least as important is the fact that various computational methods exist to carry out the inference
in practice, including MCMC methods and approximate methods like expectation propagation, Laplace approximations
and approximate Bayesian computation.
A third important aspect that appeals to users of Bayes methods is that
an implementation of a  Bayesian procedure typically produces not only an estimate of the unknown quantity of interest (usually
a posterior mean or mode), but also a large number of samples from the whole posterior distribution. These can then be used
to report a credible set, i.e.\ a set of parameter values that receives a large fixed fraction of the posterior mass, that serves as a
quantification of the uncertainty in the estimate. Some examples of papers  using Bayesian methods in nonparametric  inverse
problems in various applied settings include \cite{Gao}, \cite{Orbanz}, \cite{Lashkari}, \cite{Oh}, \cite{Bailer}. The paper \cite{Stuart}  provides a nice overview and many additional references.

Work on  the fundamental properties  of Bayes procedures for nonparametric inverse problems, like consistency, (optimal) convergence
rates, etcetera, has only started to appear  recently.  The few papers in this area include \cite{KVZ}, \cite{KVZHeat}, \cite{Simoni},
\cite{Griek}.
This is in sharp contrast with the work on frequentist methodology,
which is quite well developed. See for instance the overviews given by L.\ Cavalier \cite{Cavalier}, \cite{Cavalier2}.

Our focus in this paper is on the ability of Bayesian methods to achieve adaptive, rate-optimal inference in so-called mildly ill-posed
nonparametric inverse problems (in the terminology of, e.g., \cite{Cavalier}).
Nonparametric priors typically involve one or more tuning parameters, or hyper-parameters, that determine
the degree of regularization. In practice there is  widespread use of empirical Bayes and full, hierarchical Bayes methods
to automatically select the appropriate values of such parameters.
These methods are generally considered to be preferable to methods that use only a single, fixed value of the
hyper-parameters. In the inverse problem setting it is known from the recent paper \cite{KVZ} that using a fixed prior can indeed be
undesirable, since it can lead to convergence rates that are sub-optimal, unless by chance the statistician has
selected  a prior that  captures the fine properties of the unknown parameter (like its degree of smoothness,
if it is a function).  Theoretical work that supports the preference for empirical or hierarchical Bayes methods
does not exist at the present time however. It has until now been unknown whether these approaches can indeed robustify a procedure
against prior mismatch. In this paper we answer this question in the affirmative. We show
that empirical and hierarchical Bayes methods can lead to adaptive, rate-optimal procedures in the context of nonparametric
inverse problems, provided they are properly constructed.

We study this problem in the context of the canonical signal-in-white-noise model, or, equivalently, the
infinite-dimensional normal mean model. Using singular value decompositions many nonparametric, linear
inverse problems can be cast in this form (e.g.\ \cite{Cavalier2}, \cite{KVZ}).
Specifically, we assume that we observe a sequence of noisy coefficients $Y = (Y_1, Y_2, \ldots)$ satisfying
\begin{equation}\label{eq: ModelSeq}
Y_i = \k_i \mu_i + \frac{1}{\sqrt{n}}Z_i, \qquad i = 1, 2\ldots,
\end{equation}
where $Z_1, Z_2, \ldots$ are independent, standard normal random variables, $\mu = (\mu_1, \mu_2, \ldots) \in \ell_2$
is the infinite-dimensional parameter of interest, and $(\k_i)$ is a known sequence  that may converge to $0$ as $i \to \infty$,
which complicates the inference. We suppose the problem is  mildly ill-posed of order $p \ge 0$, in the sense that
\begin{equation}\label{eq: kappa}
C^{-1}i^{-p} \leq \k_i \leq Ci^{-p}, \qquad i = 1, 2\ldots,
\end{equation}
for some $C \geq 1$. Minimax lower bounds for the rate of convergence of estimators for $\mu$  are well known
in this setting. For instance, the lower  bound over Sobolev balls of regularity $\beta > 0$ is given by $n^{-\b/(1+2\b + 2p)}$
and over certain ``analytic balls'' the lower bound is of the order $n^{-1/2}\log^{1/2+p} n$ (see \cite{Cavalier}).
There are several regularization methods which attain these rates, including classical Tikhonov regularization
and Bayes procedures with Gaussian priors.

Many of the older existing methods for nonparametric inverse problems are not adaptive, in the sense that they  rely on knowledge of the regularity (e.g.\
in Sobolev sense) of the unknown parameter of interest to select the appropriate regularization.
This also holds  for the Bayesian approach with fixed Gaussian priors.
Early papers on the direct problem, i.e.\ the case $p=0$ in \eqref{eq: kappa},
include \cite{Zhao}, \cite{Shen}. The more recent papers
 \cite{KVZ} and \cite{Griek} study the inverse problem case, but also obtain non-adaptive results only. In the last decade however,  several methods have been developed in frequentist literature that achieve the minimax
convergence rate
without  knowledge of the regularity of the truth. This development parallels the earlier work on adaptive methods for the direct nonparametric problem
 to some extent, although the inverse case is technically usually more demanding.
The adaptive methods typically involve a data-driven choice of a tuning parameter in order
to automatically achieve an optimal bias-variance trade-off, as in Lepski's method for instance.

For nonparametric inverse problems, the construction of an adaptive estimator based on a properly penalized blockwise Stein's rule has been studied in \cite{CavTsy}, cf.\ also \cite{Cai}. This estimator is  adaptive both over Sobolev and analytic scales. In \cite{CavURE} the data-driven choice of the regularizing parameters is based on unbiased risk estimation. The authors consider projection estimators and
derive the corresponding oracle inequalities. For $\mu$ in the Sobolev scale they obtain asymptotically sharp adaptation in a minimax sense, whereas for $\mu$ in analytic  scale, their rate is optimal up to a logarithmic term.
Yet another  approach to adaptation in inverse problems is the risk hull method studied in \cite{CavRH}.
In this paper the authors consider spectral cut-off estimators  and provide oracle inequalities. An extension of their approach is presented in \cite{MarteauRH}. The link between the penalized blockwise Stein's rule and the risk hull method  is presented in \cite{MarteauHull}.

Adaptation properties of  Bayes procedures for mildly ill-posed nonparametric inverse problems have until now not been studied
in the literature. Results are only available for the direct problem, i.e.\ the case that $\k_i=1$ for every $i$, or, equivalently,
$p = 0$ in \eqref{eq: kappa}. In the paper \cite{Belitser} it is shown that in this case adaptive Bayesian inference is possible using a hierarchical,
conditionally Gaussian prior. Other recent papers also exhibit priors that yield rate-adaptive procedures in the
direct signal-in-white-noise problem (see for instance \cite{vdVvZAdaptive},  \cite{JZ}, \cite{SG}), but it is important to note that these papers use general theorems
on  contraction rates for posterior distributions (as given in \cite{GVnoniid} for instance) that are {\em not} suitable to deal with
 the truly ill-posed case in which $k_i \to 0$ as $i \to \infty$.
The reason is that if these general theorems are applied in the inverse
case, we only obtain convergence rates relative to the (squared) norm $\mu \mapsto \sum \k^2_i\mu_i^2$, which is not very interesting.
Obtaining rates relative to the $\ell_2$-norm is much more involved and requires  a different approach.
Extending the testing approach of \cite{GGvdV}, \cite{GVnoniid} would be one possibility, cf.\ the recent work of \cite{Ray}, although
it seems difficult to obtain sharp results in this manner.
In this paper we follow a more pragmatic approach, relying on partly explicit computations in a relatively tractable setting.

To obtain rate-adaptive Bayes procedures for the model \eqref{eq: ModelSeq} we consider a  family
$(\Pi_\a: \a > 0)$ of Gaussian priors for the parameter $\mu$. These priors are indexed by a parameter $\a > 0$ which
quantifies the ``regularity'' of the prior $\Pi_\a$ (details in Section \ref{sec: General}).
Instead of choosing a fixed value for $\a$ (which is the approach studied in
\cite{KVZ}) we view it as a tuning-, or hyper-parameter and consider two different methods for selecting it in a data-driven manner.
The  approach typically preferred by Bayesian statisticians is to endow the hyper-parameter with a prior distribution itself. This
results in a full, hierarchical Bayes procedure. The paper \cite{Belitser}  follows the same approach in the direct problem.
We prove that under a mild assumption on the hyper-prior on $\a$, we obtain  an adaptive procedure for the inverse problem
using the hierarchical prior.
Optimal convergence rates are obtained (up to lower order factors), uniformly over Sobolev and analytic scales.
For tractability, the priors $\Pi_\a$ that we use put independent, Gaussian prior weights on the coefficients
$\mu_i$ in \eqref{eq: ModelSeq}. Extensions to more general priors, including non-Gaussian densities or priors that are
not exactly diagonal (as in \cite{Ray} for instance) should be possible, but would require  considerable additional
technical work.

A second approach we study consists in first ``estimating'' $\a$ from the data and then substituting the estimator
$\hat\a_n$  for $\a$ in the posterior distribution for $\mu$  corresponding to the prior $\Pi_\a$.
This empirical Bayes procedure is not really Bayesian in the strict sense of the word. However, for computational
reasons empirical Bayes methods of this type are widely used in practice, making it relevant to study their theoretical performance.
Rigorous results about the asymptotic behavior of empirical Bayes selectors of hyper-parameters in infinite-dimensional
problems only exist for a limited number of special problems, see e.g.\ \cite{BelitserEB},
\cite{Zhang}, \cite{JohnstoneSilverman}, \cite{JS2005}.
In this paper we prove that the likelihood-based empirical Bayes method that we propose has the same desirable adaptation and rate-optimality properties in nonparametric inverse problems as the hierarchical Bayes approach.

The estimator $\hat\a_n$ for $\a$ that we propose is the commonly used likelihood-based empirical Bayes
estimator for  the hyper-parameter.
Concretely, it is the maximum likelihood estimator for $\a$ in the model
in which the data is generated by first drawing $\mu$ from $\Pi_\a$ and then generating $Y = (Y_1, Y_2, \ldots)$
according to \eqref{eq: ModelSeq}, i.e.\
\begin{equation}\label{eq: bs}
\mu\given\a \sim \Pi_\a, \qquad \text{and} \qquad Y\given (\mu,
\a) \sim\bigotimes_{i=1}^\infty {N}\Bigl(\k_i\mu_i,\frac{1}{n}\Bigr).
\end{equation}
A crucial element in the proof of the adaptation properties of both procedures we consider is understanding the
asymptotic behavior of $\hat\a_n$.
In contrast to the typical situation in parametric models (see \cite{JSC})
this turns out to be rather delicate, since the likelihood for $\a$
can have complicated behavior.
We are able however to derive deterministic asymptotic lower and upper bounds for $\hat\a_n$. In general these depend on the
true parameter $\mu_0$ in a very complicated way. To get some insight into  why our procedures work we show that if the true
parameter has nice regular behavior of the form $\mu_{0,i} \asymp i^{-1/2-\beta}$ for some $\beta > 0$, then
$\hat\a_n$ is essentially a consistent estimator for $\beta$ (see Lemma \ref{lem: abar}).
This means that in some sense, the estimator $\hat\a_n$ correctly ``estimates the regularity'' of the
true parameter (see \cite{BelitserEB} for  work in a similar direction).
 Since the empirical Bayes procedure basically chooses the data-dependent prior $\Pi_{\hat\a_n}$
for $\mu$, this means that asymptotically, the procedure automatically succeeds in selecting among the
priors $\Pi_\a, {\a > 0},$ the one for which the regularity of the prior and the truth are matched.
This results in an optimal bias-variance trade-off and hence in optimal convergence rates.

The remainder of the paper is organized as follows.
In Section~\ref{sec: General} we first describe the empirical and hierarchical Bayes procedures in detail.
Then we present a theorem on the asymptotic behavior of estimator $\hat\a_n$ for the hyper-parameter, followed by two results on the adaptation and rate of contraction of the empirical and hierarchical Bayes posteriors over Sobolev and analytic scales.
These results  all concern global $\ell_2$-loss.
In Section \ref{sec: functionals} we briefly comment on rates relative to other losses.
Specifically we discuss contraction rates of marginal posteriors for linear functionals
of the parameter $\mu$. We conjecture that the procedures that we prove to be adaptive and rate-optimal
for global $\ell_2$-loss, will be sub-optimal for estimating certain unbounded linear functionals.
A detailed study of this issue is outside the scope of the present paper.
The empirical and hierarchical Bayes approaches are  illustrated numerically in Section~\ref{sec: Example}. We apply them to
simulated data from an inverse signal-in-white-noise problem, where the problem is to recover
a signal from a noisy observation of its primitive.
Proofs of the main results are presented in Sections~\ref{sec: abar}--\ref{sec: ProofHB}. Some auxiliary
lemmas are collected in Section~\ref{sec: Appendix}.

\subsection{Notation}
For $\b, \g \ge 0$, the Sobolev norm  $\|\mu\|_\b$, the analytic norm $\|\mu\|_{A^\g}$ and the $\ell_2$-norm $\|\mu\|$ of
an element $\mu \in \ell_2$ are defined by
\[
\|\mu\|_\beta^2 = \sum_{i=1}^\infty i^{2\beta}\mu_i^2, \qquad \|\mu\|^2 = \sum_{i=1}^\infty \mu_i^2, \qquad \|\mu\|^2_{A^\g} = \sum_{i=1}^\infty e^{2\g i}\mu_i^2,
\]
and the corresponding Sobolev space by $S^\beta = \{\mu \in \ell_2: \|\mu\|_\beta < \infty\}$, and the analytic space by $A^\g = \{\mu \in \ell_2: \|\mu\|_{A^\g} < \infty\}$.

For two sequences $(a_n)$ and $(b_n)$ of numbers,  $a_n \asymp b_n$ means that $|a_n/b_n|$ is bounded away from zero and infinity as $n \to \infty$, $a_n \lesssim b_n$ means that $a_n/b_n$ is bounded, $a_n \sim b_n$ means that $a_n/b_n \to 1$ as $n \to \infty$, and $a_n \ll b_n$ means that $a_n / b_n \to 0$ as $n \to \infty$. For two real numbers $a$ and $b$, we denote by $a \vee b$ their maximum, and by $a \wedge b$ their minimum.

\section{Main results}\label{sec: General}

\subsection{Description of the empirical and hierarchical Bayes procedures}

We assume that we observe the sequence of noisy coefficients $Y = (Y_1, Y_2, \ldots)$ satisfying
\eqref{eq: ModelSeq}, for $Z_1, Z_2, \ldots$ independent, standard normal random variables, $\mu = (\mu_1, \mu_2, \ldots) \in \ell_2$, and a known  sequence $(\k_i)$ satisfying \eqref{eq: kappa}
for some $p \geq 0$ and $C \geq 1$. We denote the distribution of the sequence $Y$ corresponding to
 the ``true''  parameter $\mu_0$ by $\Pr_0$, and the corresponding expectation by $\E_0$.

For $\a > 0$, consider the product prior $\Pi_\a$ on $\ell_2$ given by
\begin{equation}\label{eq: prior}
\Pi_{\a}=\bigotimes _{i=1}^{\infty}N\bigl(0,i^{-1-2\a}\bigr).
\end{equation}
It is easy to see that this prior is ``$\a$-regular'', in the sense that for every $\a' < \a$, it assigns mass $1$ to the
Sobolev space $S^{\a'}$. In \cite{KVZ} it was proved that if for the true parameter $\mu_0$ we have $\mu_0 \in S^\b$
for $\b > 0$, then the posterior distribution corresponding to the Gaussian prior $\Pi_\a$ contracts around $\mu_0$
at the optimal rate $n^{-\b/(1+2\b + 2p)}$ if $\a =\b$. If $\a \not = \b$, only sub-optimal rates are  attained in general (cf.\ \cite{Ismael}).
In other words, when using a Gaussian prior with a fixed regularity, optimal convergence rates are obtained if and only if the
regularity of the prior and the truth are matched. Since the latter is unknown however, choosing the prior that is
optimal from the point of view of convergence rates is typically not possible in practice.
Therefore,  we consider two data-driven methods
for selecting the regularity of the prior.

The first is a likelihood-based empirical Bayes method, which attempts to estimate the appropriate value
of the hyper-parameter $\a$ from the data.
 In the Bayesian setting described by the conditional distributions (\ref{eq: bs}),
it holds that
\[
Y \given \a \sim \bigotimes_{i=1}^{\infty}{N}\Bigl(0,i^{-1-2\a}\k_i^{2}+ \frac{1}{n}\Bigr).
\]
The corresponding log-likelihood for $\a$ (relative to an infinite
product of $N(0,1/n)$-distributions) is easily seen to be given by
\begin{equation}\label{eq: ell}
\ell_n(\a)=-\frac{1}{2}\sum_{i=1}^{\infty}\Big( \log\Big(1+\frac{n}{i^{1+2\a}\k_i^{-2}}\Big)-\frac{n^2} {i^{1+2\a}\k_i^{-2}+n}Y_i^2 \Big).
\end{equation}
The idea is to ``estimate'' $\a$ by the maximizer of $\ell_n$.
The results ahead (Lemma \ref{lem: abar} and Theorem \ref{thm: AlphaMagnitude}) imply that  with $\Pr_0$-probability tending to one, $\ell_n$ has a global maximum
on $[0, \log n)$ if  $\mu_{0, i} \not = 0$ for some $i \ge 2$.
(In fact, the cited results imply the maximum  is attained on the slightly smaller interval $[0, (\log n)/(2\log 2) - 1/2-p]$).
If the latter condition is not satisfied (if $\mu_0 = 0 $ for instance),
$\ell_n$ may attain its maximum only at $\infty$. Therefore, we truncate the maximizer at $\log n $ and define
\[
\hat \a _n = \argmax_{\a \in [0, \log n]} \ell_n(\a).
\]
The continuity of  $\ell_n$ ensures the  $\argmax$  exists. If it  is not unique, any value may be chosen.
We will always assume at least that $\mu_0$ has Sobolev regularity of some order $\beta > 0$.
Lemma \ref{lem: abar} and Theorem \ref{thm: AlphaMagnitude} imply that in this case
$\hat\a_n > 0$ with probability tending to $1$.
An alternative to the truncation of the argmax of $\ell_n$ at $\log n$ could be to extend the definition of
the priors $\Pi_\a$ to include the case $\a =\infty$. The prior $\Pi_\infty$ should then be defined as the
product $N(0,1) \otimes \delta_0 \otimes \delta_0 \otimes \cdots$, with $\delta_0$ the Dirac measure
concentrated at $0$. However, from a practical perspective
it is more convenient to define $\hat\a_n$ as above.

The empirical Bayes procedure  consists in computing the posterior distribution of $\mu$ corresponding to a
fixed prior $\Pi_\a$ and then substituting $\hat\a_n$ for $\a$.
Under the model described above and the prior \eqref{eq: prior} the coordinates $(\mu_{0,i}, Y_i)$ of the vector $(\mu_0, Y)$ are independent, and hence the conditional distribution of $\mu_0$ given $Y$ factorizes over the coordinates as well. The computation of the posterior distribution reduces to countably many posterior computations in conjugate normal models. Therefore (see also \cite{KVZ}) the posterior distribution corresponding to the prior $\Pi_\a$ is given by
\begin{equation}\label{eq: PostDist}
\Pi_\a(\, \cdot\, \given Y) = \bigotimes_{i=1}^{\infty}{N}\Big(\frac{n\k_i^{-1}}{i^{1+2\a}\k_i^{-2} + n}Y_i,
\frac{\k_i^{-2}}{i^{1+2\a}\k_i^{-2}+n}\Big).
\end{equation}
Then the {\em empirical Bayes posterior} is  the random measure $\Pi_{\hat\a_n}(\, \cdot\,  \given Y)$ defined by
\begin{equation}\label{eq: ebpost}
\Pi_{\hat\a_n}(B \given Y) =
\Pi_{\a}(B \given Y) \Big|_{\a = \hat\a_n}\\
\end{equation}
for measurable subsets $B \subset \ell_2$.
Note that the construction of the empirical Bayes posterior does not use information about the regularity of the true parameter.
In Theorem \ref{thm: ConvergenceEB} below we prove that it contracts around the truth at an optimal rate (up to lower order factors),
uniformly over Sobolev and analytic scales.


The  second method we consider is a full, hierarchical Bayes approach where we put a prior
distribution on the hyper-parameter $\a$. We use a prior on $\a$ with a positive Lebesgue density $\l$ on $(0,\infty)$.
The full, hierarchical prior for $\mu$ is then given by
\begin{equation}\label{eq: hp}
\Pi = \int_{0}^\infty \l(\a)\Pi_\a\, d\a.
\end{equation}
In Theorem \ref{thm: ConvergenceHB} below we prove that under mild assumptions on the prior density $\l$,
the corresponding posterior distribution $\Pi(\, \cdot  \, \given Y)$ has the same desirable
asymptotic properties as the empirical Bayes posterior \eqref{eq: ebpost}.

\subsection{Adaptation and contraction rates for the full parameter}
\label{sec: main}

Understanding of the asymptotic behavior of the maximum likelihood estimator $\hat\a_n$
is a crucial element in our proofs of the contraction rate results for the empirical and hierarchical Bayes procedures.
The estimator somehow ``estimates'' the regularity of the true parameter $\mu_0$, but in a rather indirect and involved manner
in general.
Our first theorem gives deterministic upper and lower bounds for $\hat\a_n$, whose construction involves  the
function   $h_n:(0,\infty)\to [0,\infty)$ defined by
\begin{equation}\label{eq: h}
h_n(\a)=\frac{1+2\a+2p}{n^{1/(1+2\a+2p)}\log n}\sum_{i=1}^{\infty}\frac{n^2i^{1+2\a} \mu_{0,i}^2\log i}{(i^{1+2\a}\k_i^{-2}+n)^2}.
\end{equation}
For positive constants $0 < l<L$ we define the  lower and upper bounds as
\begin{align}
\underline{\a}_n&=\inf\{\a>0: h_n(\a)>l\}\wedge\sqrt{\log n}\label{eq: LB}, \\
\overline{\a}_n&=\inf\{\a>0: h_n(\a)>L(\log n)^2\}\label{eq: UB}.
\end{align}

One can see that the function $h_n$ and hence the lower and upper bounds $\underline\a_n$ and $\overline\a_n$ depend on
the true $\mu_0$. We show in Theorem \ref{thm: AlphaMagnitude} that the maximum likelihood estimator $\hat\a_n$ is between these bounds with probability tending to one. In general the true $\mu_0$ can have very complicated tail behavior,
which makes it difficult to understand the behavior of the
upper and lower bounds. If $\mu_0$ has regular tails however, we can get some insight in
the nature of the bounds. We have the following lemma, proved in Section \ref{sec: abar}.

\begin{lemma}\label{lem: abar}
For any $l, L > 0$ in the definitions \eqref{eq: LB}--\eqref{eq: UB} the following statements hold.
\begin{enumerate}
\item[(i)]
For all $\beta , R > 0$, there exists $c_0 > 0$ such that
\[
\inf_{\|\mu_0\|_\beta \le R} \underline\a_n \ge \beta -\frac{c_0}{\log n}
\]
for $n$ large enough.
\item[(ii)]
For all $\gamma , R > 0$,
\[
\inf_{\|\mu_0\|_{A^\g} \le R} \underline\a_n \ge \frac{\sqrt{\log n}}{\log\log n}
 \]
 for $n$ large enough.
\item[(iii)]
If $\mu_{0,i}\geq ci^{-\g-1/2}$ for some $c, \g > 0$, then for a constant $C_0 > 0$ only depending on $c$ and $\g$,
we have $\overline{\a}_n \leq \g + C_0({\log\log n})/{\log n}$ for all $n$ large enough.
\item[(iv)]
If $\mu_{0,i} \not = 0$ for some $i \ge 2$, then $\bar\a_n \le  (\log n)/(2\log 2) - 1/2- p$ for $n$ large enough.
\end{enumerate}
\end{lemma}

We note that items (i) and (iii) of the lemma imply that if
$\mu_{0,i} \asymp i^{-1/2-\beta}$, then the interval $[\underline\a_n, \bar\a_n]$ concentrates around
the value $\beta$ asymptotically. In combination with Theorem \ref{thm: AlphaMagnitude} this shows that at least in this regular case,
$\hat\a_n$ correctly estimates the regularity of the truth. The same is true in the analytic case, since item
(ii) of the lemma shows that $\underline\a_n \to \infty$ in that case, i.e.\ asymptotically, the procedure
detects the fact that $\mu_0$ has infinite regularity.

Item (iv) implies  that if $\mu_{0,i} \not = 0$ for some $i \ge 2$, then $\bar\a_n < \infty$ for large $n$.
Conversely, the definitions of $h_n$ and $\bar\a_n$ show that if $\mu_{0,i} = 0$ for
all $i \ge 2$, then $h_n \equiv 0$ and hence $\bar\a_n = \infty$.

The following theorem asserts that the point(s) where $\ell_n$ is maximal is (are) asymptotically between the bounds just defined,
uniformly over Sobolev and analytic scales.
The proof is given in Section \ref{sec: AlphaMagnitude}.

\begin{theorem}\label{thm: AlphaMagnitude}
For every $R > 0$
the constants $l$ and $L$ in \eqref{eq: LB} and \eqref{eq: UB} can be chosen such that
\[
\inf_{\mu_0\in\B(R)} \Pr_0 \Big(\argmax_{\a \in [0,\log n]} \ell_n(\a) \in [\underline{\a}_n, \overline{\a}_n]\Big) \to 1,
\]
where
$\B(R) = \{\mu_0\in\ell_2:\|\mu_0\|_\b\leq R\}$ or $\B(R) = \{\mu_0\in\ell_2:\|\mu_0\|_{A^\g}\leq R\}$.
\end{theorem}

With the help of Theorem~\ref{thm: AlphaMagnitude} we can prove the following theorem, which states that the empirical Bayes posterior distribution \eqref{eq: ebpost} achieves optimal minimax contraction rates up to a slowly varying factor,  uniformly
 over Sobolev and analytic scales. We note that posterior contraction at a rate $\e_n$
 implies the existence of estimators, based on the posterior, that converge at the same rate. See for instance
 the construction in Section 4 of \cite{Belitser}.

\begin{theorem}\label{thm: ConvergenceEB}
For every $\beta, \gamma, R > 0$ and $M_n \to \infty$ we have
\[
\sup_{\|\mu_0\|_\beta \le R} \E_0 \Pi_{\hat{\a}_n}\bigl( \|\mu-\mu_0\| \geq M_nL_nn^{-\b/(1+2\b+2p)}\, \big|\, Y\bigr) \to 0
\]
and
\[
\sup_{\|\mu_0\|_{A^\g} \le R} \E_0 \Pi_{\hat{\a}_n}\bigl( \|\mu-\mu_0\| \geq M_nL_n(\log n)^{1/2+p}n^{-1/2}\, \big|\, Y\bigr) \to 0,
\]
where $(L_n)$ is a slowly varying sequence.
\end{theorem}

So indeed we see that both in the Sobolev and analytic cases, we obtain
the optimal minimax rates up to a slowly varying factor. The proofs of the statements (given in  Section \ref{sec: ProofSob})
show that in the first case we can take $L_n = (\log n)^{2}(\log\log n)^{1/2}$ and in the second case
$L_n = (\log n)^{(1/2+p)\sqrt{\log n}/2+1-p}(\log\log n)^{1/2}$.
These sequences converge to infinity but they are slowly varying, hence they converge slower than any power of $n$.

%

%

The full Bayes procedure using the hierarchical prior \eqref{eq: hp} achieves
the same results as the empirical Bayes method, under mild assumptions on the prior density $\l$ for $\a$.

\begin{assumption}\label{ass: Hyperprior}
%
{
Assume that for every $c_1 > 0$ there exist $c_2 \geq 0$, $c_3 \in \RR$, with $c_3>1$ if $c_2=0$, and $c_4>0$ such that
\[
c_4^{-1}\a^{-c_3} \exp(-c_2\a) \leq \l(\a) \leq c_4\a^{-c_3} \exp(-c_2\a)
\]
for $\a \geq c_1$.
}
\end{assumption}

One can see that a many distributions satisfy this  assumption,
{for instance the exponential, gamma and inverse gamma distributions.}
Careful inspection of the proof of the following theorem, given in Section \ref{sec: ProofHB}, can lead to weaker assumptions, although these will be less attractive to formulate. Recall the notation $\Pi(\,\cdot\,\given Y)$ for the posterior corresponding to the
hierarchical prior (\ref{eq: hp}).

\begin{theorem}\label{thm: ConvergenceHB}
Suppose the prior  density $\l$ satisfies Assumption~\ref{ass: Hyperprior}.
Then for every $\beta, \gamma, R > 0$ and $M_n \to \infty$ we have
\[
\sup_{\|\mu_0\|_\beta \le R} \E_0 \Pi\bigl( \|\mu-\mu_0\| \geq M_nL_nn^{-\b/(1+2\b+2p)}\, \big|\, Y\bigr) \to 0
\]
and
\[
\sup_{\|\mu_0\|_{A^\g} \le R} \E_0 \Pi\bigl( \|\mu-\mu_0\| \geq M_nL_n(\log n)^{1/2+p}n^{-1/2}\, \big|\, Y\bigr) \to 0,
\]
where $(L_n)$ is a slowly varying sequence. \end{theorem}

The hierarchical Bayes method thus yields exactly the same rates as the empirical method, and therefore the interpretation of this theorem is the same as before.
We note that already in the direct case $p=0$ this theorem is an interesting extension
of the existing results of \cite{Belitser}. In particular we find that using hierarchical Bayes we can adapt to a continuous
range of Sobolev regularities while incurring only a logarithmic correction of the optimal rate.


\subsection{Discussion on linear functionals}
\label{sec: functionals}

It is known already in the non-adaptive situation that for attaining optimal rates
relative to losses other than the $\ell_2$-norm, it may  be necessary to set the hyperparameter to a value
different from the optimal choice for $\ell_2$-recovery of the full parameter $\mu$.
If we are for instance interested in optimal estimation of the (possibly unbounded)
linear functional
\begin{equation}\label{eq: l}
L\mu = \sum l_i \mu_i,
\end{equation}
where $l_i \asymp i^{-q-1/2}$ for some $q < p$, then if $\mu_0 \in S^\beta$ for $\beta > -q$ the optimal Gaussian prior \eqref{eq: prior} is not
$\Pi_\beta$, but rather $\Pi_{\beta-1/2}$. The resulting, optimal rate is of the order $n^{-(\beta+q)/(2\beta+2p)}$
(see \cite{KVZ}, Section 5).

An example of this phenomenon occurs when considering global $L_2$-loss estimation of a function
versus pointwise estimation.
If for instance the $\mu_i$ are the Fourier coefficients of a smooth function of interest $f \in L^2[0,1]$
relative to the standard Fourier basis $e_i$ and for a fixed $t \in [0,1]$, $l_i = e_i(t)$,
then estimating $\mu$ relative to $\ell_2$-loss corresponds to estimating $f$ relative to $L_2$-loss and
estimating the functional $L\mu$ in \eqref{eq: l} corresponds to pointwise estimation of $f$ in the point $t$
(in this case $q = -1/2)$.

Theorems \ref{thm: ConvergenceEB} and \ref{thm: ConvergenceHB} show that
the empirical and hierarchical Bayes procedures automatically achieve a
bias-variance-posterior spread trade-off that is optimal for the recovery
of the full parameter $\mu_0$ relative to the global $\ell_2$-norm.
As conjectured in a similar setting in \cite{KVZ} this suggests that the adaptive approaches might be sub-optimal
outside the $\ell_2$-setting. In view of the findings in the non-adaptive case
we might expect however that we can slightly alter the procedures to deal with linear functionals.
For instance, it is natural to expect that for the linear functional
\eqref{eq: l}, the empirical Bayes posterior $\Pi_{\hat\a_n-1/2}(\cdot\given Y)$ yields  optimal rates.

Matters seem to be more delicate however.
A combination of  elements of the proof of Theorem 5.1 of
\cite{KVZ} and new results on the coverage of credible sets from the forthcoming paper
\cite{SVZ} lead us to conjecture that for every $\beta > -q$
there exists a $\theta_0 \in S^\beta$ such that along a subsequence $n_j$,
\begin{align*}
\E_{0}\Pi_{\hat\a_{n_j}-1/2}\bigl(\mu:\, |L\mu_0-L\mu|\geq m n_j^{-(\beta+q)/(1+2\beta+2p)}\given Y\bigr) \to 1,
\end{align*}
as $j \to \infty$ for a positive, small enough constant $m > 0$.
In other words,  there always exist ``bad truths'' for which the adjusted empirical Bayes procedure
converges at a sub-optimal rate along a subsequence.
For linear functionals  \eqref{eq: l}
the empirical Bayes posterior $\Pi_{\hat\a_n-1/2}(\cdot\given Y)$ seems only to contract
at an optimal rate for ``sufficiently nice'' truths, for instance of the form $\mu_{0, i} \asymp i^{-1/2-\beta}$.

Similar statements are expected to hold for hierarchical Bayes procedures. This adds
to the list of remarkable behaviours of marginal posteriors for linear functionals, cf.\
also \cite{Rivoirard}, for instance. Further research is necessary to shed more light on these matters.

\section{Numerical illustration}\label{sec: Example}

Consider the inverse signal-in-white-noise problem where we observe the process
$(Y_t: t\in [0,1])$ given by
\[
Y_t = \int_0^t\int_0^s\mu(u)\, du\, ds + \frac1{\sqrt n}W_t,
\]
with $W$ a standard Brownian motion, and the aim is to recover the function $\mu$.
If, slightly abusing notation,  we define $Y_i = \int_0^1 e_i(t)\,dY_t$, for $e_i$ the orthonormal basis functions
given by $e_i(t) = \sqrt{2}\cos((i-1/2)\pi t)$, then it is easily verified that the observations
$Y_i$ satisfy (\ref{eq: ModelSeq}), with
$\k_i^2 = ({(i-1/2)^2\pi^2})^{-1}$, i.e.\ $p=1$ in \eqref{eq: kappa}, and $\mu_i$ the Fourier
coefficients of $\mu$ relative to the basis $e_i$.

We consider simulated data from this model for $\mu_0$ the function with
Fourier coefficients $\mu_{0,i} = i^{-3/2}\sin(i)$, so we have a truth which essentially has regularity $1$.
In the following figure we plot the true function $\mu_0$ (black curve) and the empirical Bayes posterior mean (red curve) in the left panels, and the corresponding normalized likelihood
$\exp({\ell_n})/\max(\exp({\ell_n}))$ in the right panels (we truncated the sum in (\ref{eq: ell}) at a high level). Figure~\ref{fig: 1}
shows the results for the empirical Bayes procedure with simulated data for
 $n= 10^3, 10^5, 10^7, 10^9$, and $10^{11}$, from  top to  bottom.
The figure shows that the estimator $\hat\a_n$ does a good job in this case at estimating
the regularity level $1$, at least for large enough $n$. We also see however that due to the ill-posedness
of the problem, a large signal-to-noise ratio $n$ is necessary for accurate recovery of the function $\mu$.


\begin{figure}
\centerline{\includegraphics[width=7.5cm]{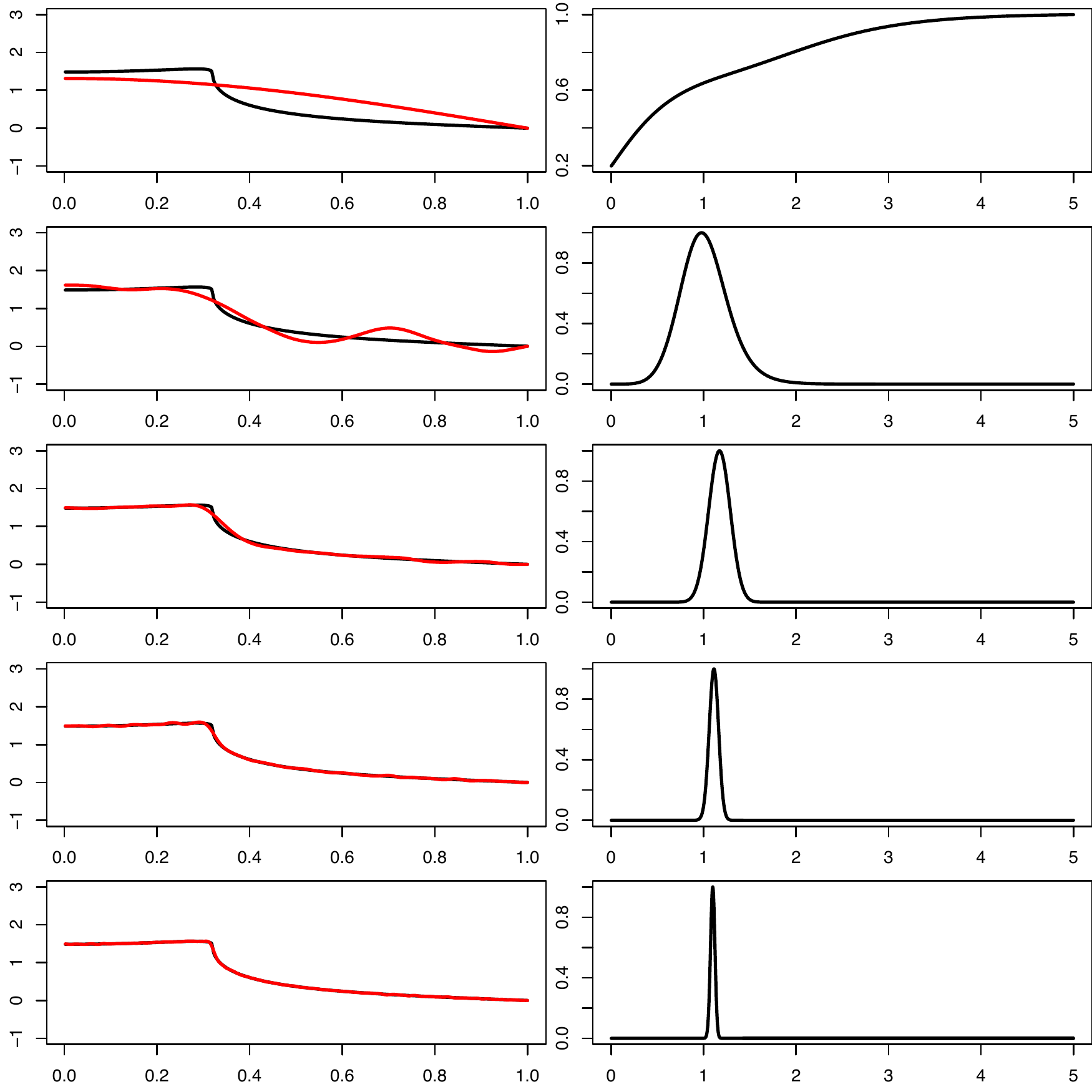}}
\caption{Left panels:  the empirical Bayes posterior mean (red) and the true curve (black). Right panels: corresponding normalized likelihood for $\a$. We have $n= 10^3, 10^5, 10^7, 10^9$, and $10^{11}$, from top to bottom.}
\label{fig: 1}
\end{figure}

%
%
%

We  applied the  hierarchical Bayes method to the simulated data as well.
We chose a standard exponential prior distribution on $\a$, which satisfies Assumption \ref{ass: Hyperprior}.
Since the posterior can not be computed explicitly, we implemented an
MCMC algorithm that generates (approximate) draws from the posterior distribution of
 the pair $(\a, \mu)$. More precisely, we  fixed a large index $J \in \NN$
 and defined the vector
 $\mu^J = (\mu_1, \ldots, \mu_J)$  consisting of the first $J$ coefficients of $\mu$.
(If $\mu$ has positive Sobolev regularity, then taking  $J$ at least of the order $n^{1/(1+2p)}$
ensures that the approximation error $\|\mu^J - \mu\|$ is of lower order than the
estimation rate.)
 Then we
devised a Metropolis-within-Gibbs  algorithm for sampling from the posterior distribution of $(\a, \mu^J)$
(e.g.\ \cite{Tier}).
The algorithm  alternates between draws from the conditional distribution $\mu^J \given \a, Y$  and the conditional distribution $\a \given \mu^J, Y$. The former is explicitly given by \eqref{eq: PostDist}. To sample from
$\a \given \mu^J, Y$ we used a standard Metropolis-Hastings step.  It is easily verified that the  Metropolis-Hastings
acceptance probability for  a move from $(\a, \mu)$ to $(\a', \mu)$ is given by
 \[
 1 \wedge \frac{q(\a'\given \a)p(\mu^J\given \a')\l(\a')}{q(\a\given \a')p(\mu^J\given \a)\l(\a)},
 \]
 where $p(\,\cdot\,\given\a)$ is the density of $\mu^J$ if $\mu \sim \Pi_\a$, i.e.\
\[
p(\mu^J\given\a) \propto \prod_{j=1}^J j^{1/2+\a}e^{-\frac12j^{1+2\a}\mu^2_j},
\]
and $q$ is the transition kernel of the proposal chain. We used a proposal chain that, if it is currently at location $\a$, moves to a new $N(\a, \s^2)$-distributed location provided the latter is positive.  We omit further details,  the implementation is straightforward.

The results for the hierarchical Bayes procedure are given in Figure \ref{fig: 5}.
The figure shows the results for simulated data with $n=10^3,10^5,10^7,10^9$ and $10^{11}$, from top to bottom.
Every time we  see  the posterior mean (in  red) and the true curve (black) on the left
and a histogram for the posterior of  $\a$ on the right. The results are  comparable to what we found for the
empirical Bayes procedure.

\begin{figure}
\centerline{\includegraphics[width=7.5cm]{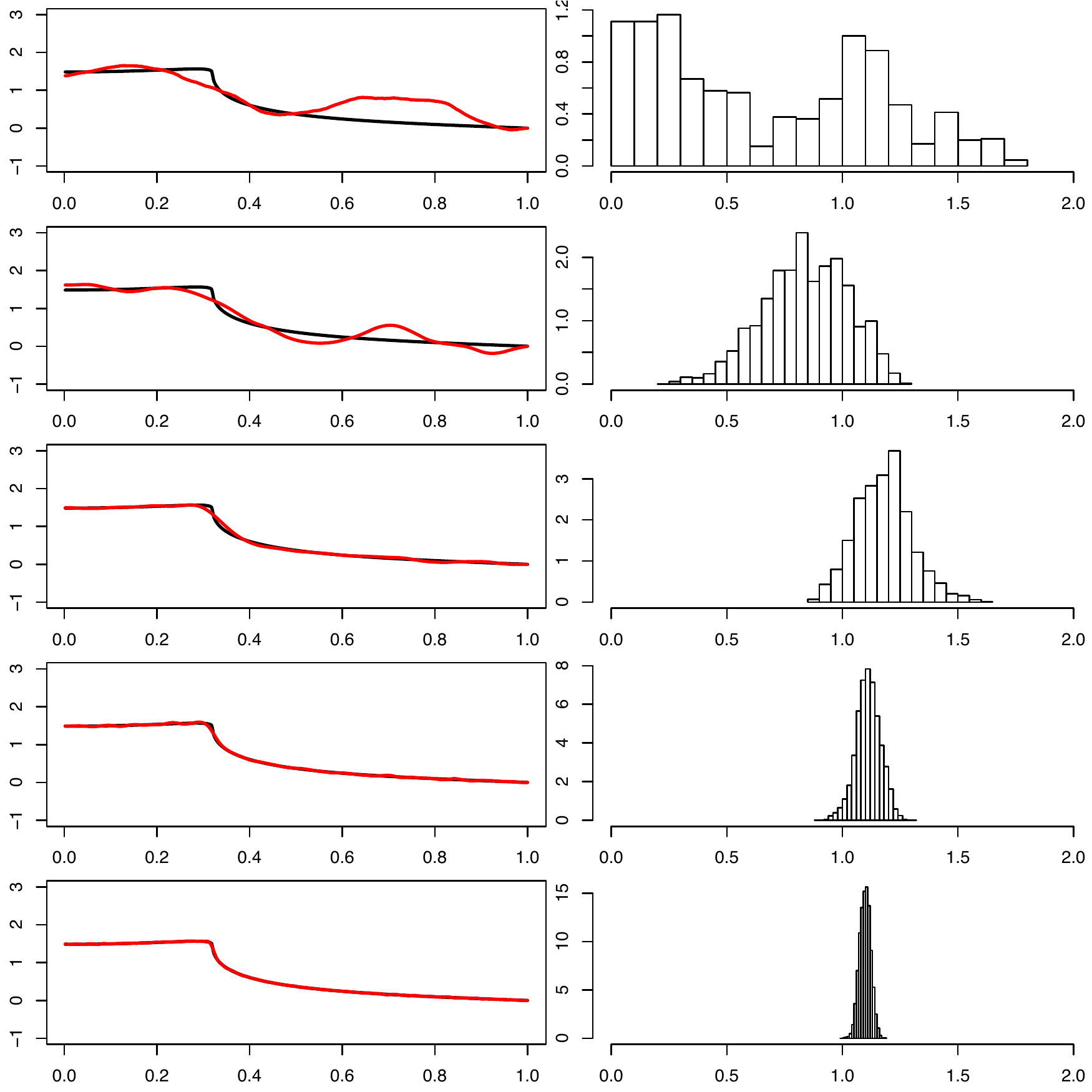}}
\caption{Left panels: the hierarchical Bayes posterior mean (red) and the true curve (black).
Right panels: histograms of posterior for $\a$. We have $n= 10^3, 10^5, 10^7, 10^9$, and $10^{11}$  from top to bottom.}
\label{fig: 5}
\end{figure}

\section{Proof of Lemma \ref{lem: abar}}\label{sec: abar}
In the proofs we assume for brevity that we have the exact equality $\kappa_ i = i^{-p}$.
Dealing with the general case  \eqref{eq: kappa} is straightforward, but makes the proofs
somewhat lengthier.

(i).  We show that for all
$\a \leq \b - c_0/\log n$, for some large enough constant $c_0 > 0$
that only depends on $l$, $\beta, \|\mu_0\|_\beta$ and $p$, it holds that
$h_n(\a) \le l$, where $l$ is the given positive constant in the definition of $\underline{\a}_n$.

The sum in the definition \eqref{eq: h} of $h_n$ can be split into two sums, one
over indices $i \leq n^{1/(1+2\a + 2p)}$ and one over indices $i > n^{1/(1+2\a + 2p)}$.
The second sum   is bounded by
\[
 n^2\sum_{i \ge n^{1/(1+2\a+2p)}} i^{-1-2\a-4p-2\b}(\log i)i^{2\b}\mu_{0,i}^2.
\]
Since the function $x \mapsto x^{-\g}\log x$ is decreasing on $[e^{1/\g}, \infty)$,
this is further bounded by
\[
\frac{\|\mu_0\|_\beta^2}{1+2\a+2p} n^{\frac{1+2\a-2\b}{1+2\a+2p}}{\log n}.
\]
The sum over  $i \leq n^{1/(1+2\a + 2p)}$  is upper bounded by
\[
\sum_{i \leq n^{1/(1+2\a + 2p)}}i^{1+2\a-2\b}i^{2\b}\mu_{0,i}^2\log i.
\]
Since the logarithm is  increasing we can  take $(\log n)/(1+2\a+2p)$ outside the sum and then bound $i^{1+2\a-2\b}$ above by $n^{(1+2\a-2\b)/(1+2\a+2p) \vee 0}$ to arrive at the subsequent bound
\[
\frac{\|\mu_0\|_\beta^2}{1+2\a+2p} n^{0 \vee \frac{1+2\a-2\b}{1+2\a+2p}}{\log n}.
\]
Combining the  bounds for the two sums we obtain the upper bound
\[
h_n(\a) \le \|\mu_0\|_\beta^2 n^{-\frac{1\wedge 2(\b-\a)}{1+2\a+2p}},
\]
valid for all $\a > 0$.
Now suppose that  $\a \leq \b - c_0/\log n$. Then for $n$ large enough,
the power of $n$ on the right-hand side is bounded by
\[
n^{-\frac{1\wedge 2(c_0/\log n)}{1+2\b+2p}} =  e^{-\frac{2c_0}{1+2\b+2p}}.
\]
Hence given $l > 0$ we can  choose $c_0$ so large,  only depending on $l$, $\beta, \|\mu_0\|_\beta$ and $p$,
 that $h_n(\a) \le l$ for $\a \leq \b - c_0/\log n$.

(ii). We  show that  in this case we have $h_n(\a) \le l$ for $\a \leq \sqrt{\log n}/(\log\log n)$
and $n \ge n_0$, where $n_0$ only depends on $\|\mu_0\|_{A^\g}$.
Again we give an upper bound for $h_n$ by splitting the sum in its definition into two smaller sums. The one
over indices  $i > n^{1/(1+2\a + 2p)}$ is bounded by
\[
n^2\sum_{i> n^{1/(1+2\a+2p)}} i^{-1-2\a-4p}e^{-2\g i}(\log i)e^{2\g i}\mu_{0,i}^2.
\]
Using the fact that for $\d > 0$ the function $x \mapsto x^{-\d}e^{-2\g x}\log x$  is
decreasing on  $[e^{1/\d}, \infty)$ we can see that this is further bounded by
\[
\frac{\|\mu_0\|_{A^\g}^2}{1+2\a+2p} e^{-2\g n^{1/(1+2\a+2p)}} n^{\frac{1+2\a}{1+2\a+2p}}{\log n}.
\]
The sum over indices $i \leq n^{1/(1+2\a + 2p)}$ is bounded by
\[
\frac{\log n}{1+2\a+2p}\sum_{i \le n^{1/(1+2\a+2p)}}i^{1+2\a}e^{-2\g i}e^{2\g i}\mu_{0,i}^2.
\]
Since the maximum on $(0,\infty)$ of the function $x \mapsto x^{1+2\a}\exp(-2\g x)$ equals
$\exp((1+2\a)(\log((1+2\a)/2\g) - 1))$, we have the subsequent bound
\[
\frac{\|\mu_0\|_{A^\g}^2}{1+2\a+2p} e^{(1+2\a)\log((1+2\a)/2\g)}{\log n}.
\]
Combining the two bounds we find that
\[
h_n(\a) \le \|\mu_0\|_{A^\g}^2 \Big(n^{\frac{2\a}{1+2\a+2p}}e^{-2\g n^{\frac1{1+2\a+2p}}}
+ n^{-\frac1{1+2\a+2p}}e^{(1+2\a)\log\frac{1+2\a}{2\g}}\Big)
\]
for all $\a > 0$. It is then easily verified that for the given constant $l > 0$,
we have $h_n(\a) \le l$ for $n \ge n_0$ if $\a \le \sqrt{\log n} /\log\log n$, where $n_0$
	only depends on $\|\mu_0\|_{A^\g}$.

(iii).
Let $\g_n = \g + C_0(\log\log n)/(\log n)$.
We will  show that for $n$ large enough, $h_n(\g_n) \geq L(\log n)^2$, provided $C_0$ is large enough.
Note that
\begin{equation*}
\sum_{i=1}^\infty \frac{n^2i^{1+2\g_n}\mu_{0,i}^2\log i}{(i^{1+2\g_n+2p}+n)^2}
\geq \frac{c^2}{4}\sum_{i \le n^{1/(1+2\g_n+2p)}} i^{2(\g_n-\g)}\log i.
\end{equation*}
By monotonicity and the fact that $\lfloor x\rfloor  \ge x/2$ for $x$ large,  the sum on the right is bounded from below by
the integral
\begin{align*}
\int_{0}^{n^{1/(1+2\g_n+2p)}/2} x^{2\g_n-2\g}\log x\, dx.
\end{align*}
This integral can be computed explicitly and is for large $n$ bounded from below by a constant times
\[
\frac{\log n}{1+2\g_n + 2p} n^{\frac{2\g_n-2\g +1}{1+2\g_n+2p}}.
\]
It follows that, for large enough $n$, $h_n(\g_n)$ is bounded from below by a constant times $c^2 n^{{2(\g_n-\g)}/{(1+2\g_n+2p)}}$.
Since $(\log\log n)/(\log n) \leq 1/4$ for  $n$ large enough, we obtain
\[
n^{2(\g_n-\g)/(1+2\g_n+2p)} \geq n^{\frac{1}{\log n}(\log\log n)\frac{2C_0}{1+2\g+C_0/2+2p}} = (\log n)^{2C_0/(1+2\g+C_0/2+2p)}.
\]
Hence for  $C_0$ large enough, only depending on $c$ and $\g$, we indeed have that and $h_n(\g_n) \geq L(\log n)^2$ for large $n$.

(iv). If $\mu_{0,i} \not = 0$ for $i \ge 2$, then
\[
h_n(\a) \gtrsim \frac{1+2\a+2p}{n^{1/(1+2\a+2p)}\log n} \frac{n^2i^{1+2\a}}{(i^{1+2\a+2p} + n)^2}.
\]
Now define $\a_n$ such that   $i^{1+2\a_n+2p} = n$. Then  by construction we have
$h_n(\a_n) \gtrsim n^{1-(1+2p)/(1+2\a_n+2p)}$. Since $\a_n \to \infty$  the right side is larger than $L\log^2 n$ for
 $n$ large enough, irrespective of the value of $L$,  hence $\bar\a_n \le \a_n \le  (\log n)/(2\log 2) -1/2-p$.

\section{Proof of Theorem~\ref{thm: AlphaMagnitude}}\label{sec: AlphaMagnitude}

With the help of the dominated convergence theorem one can see that the random function $\ell_n$ is $(\Pr_0-a.s.)$ differentiable and its derivative, which we denote by  $\MM_n$, is given by
\begin{equation*}\label{eq: m}
\MM_n(\a)  =
\sum_{i=1}^{\infty}\frac{n\log i}{i^{1+2\a}\k_i^{-2}+n}
-\sum_{i=1}^{\infty}\frac{n^2i^{1+2\a}\k_i^{-2}\log i}{(i^{1+2\a}\k_i^{-2}+n)^2}Y_i^2.
\end{equation*}

We will show that on the interval $(0,\underline\a_n+1/\log n]$ the random function $\MM_n$ is positive and bounded
away from $0$ with probability tending to one, hence $\ell_n $ has  no local maximum in this interval.
Next we distinguish two cases according to the value of $\overline\a_n$. If $\overline{\a}_n > \log n$, then
the inequality $\hat\a_n\leq\overline{\a}_n$ trivially  holds. In  the case $\overline\a_n\leq \log n$
we show that for a constant $C_1 > 0$ we a.s.\ have
\begin{equation}\label{eq: mhelp}
\ell_n(\a) - \ell_n(\overline\a_n) = \int_{\overline\a_n}^\a \MM_n(\gamma)\,d\gamma
\le C_1 \frac{n^{1/(1+2\overline{\a}_n+2p)}(\log n)^2}{1+2\overline{\a}_n+2p}
\end{equation}
for all $\a \ge \overline \a_n$.
Then we prove that for any given $C_2 > 0$, the constant $L$ can be set such  that for
$\gamma \in [\overline{\a}_n-1/\log n,\overline{\a}_n]$ we have
\[
\MM_n(\gamma) \le -C_2\frac{n^{1/(1+2\overline{\a}_n+2p)}(\log n)^3}{1+2\overline{\a}_n+2p}
\]
with probability tending to one uniformly. Together with \eqref{eq: mhelp} this means that on the interval $[\overline{\a}_n-1/\log n,\overline{\a}_n]$ the  function $\ell_n$ decreases more than it can possibly increase on the interval $[\overline{\a}_n,\infty)$. Therefore,
it holds with probability tending to one that $\ell_n$ has no global maximum on  $(\overline{\a}_n-1/\log n, \infty)$.

We only present the details of the proof for the case that $\mu_0 \in S^\beta$. The case $\mu_0 \in A^\gamma$ can
be handled along the same lines. Again for simplicity we assume $\kappa_i = i^{-p}$ in the proof.

\subsection{$\MM_n(\a)$ on $[\overline{\a}_n,\infty)$}\label{sec: OverA}

In this section we give a deterministic upper bound for the integral of $\MM_n(\a)$ on the interval $[\overline{\a}_n,\infty)$.

We have the trivial bound
\[
\MM_n(\a) \le \sum_{i=1}^{\infty}\frac{n\log i}{i^{1+2\a+2p}+n}.
\]
An application of  Lemma~\ref{lem: LogSeries}.(i) with $r = 1 +2\a + 2p$ and $c = \beta +2p$ shows that
for $\beta/2 < \a \le \log n$,
\[
\MM_n(\a) \lesssim \frac{1}{1+2\a+2p} n^{1/(1+2\a+2p)} \log n.
\]
For  $\a \ge \log n$ we apply Lemma~\ref{lem: LogSeries}.(ii), and see that $\MM_n(\a) \lesssim n2^{-1-2\a-2p}$.
Using the fact that  $x \mapsto 2^{-x} x^3$ is decreasing for large $x$, it is easily seen that
$n2^{-1-2\a-2p} \lesssim (\log n)^3/(1+2\a + 2p)^3$ for $\a \ge \log n$, hence
\[
\MM_n(\a) \lesssim \frac{(\log n)^3}{(1+2\a + 2p)^3}.
\]
By Lemma \ref{lem: abar} we have $\b/2 < \overline{\a}_n$ for large enough $n$. It follows that the integral
we want to bound is bounded by a constant times
\[
n^{1/(1+2\bar\a_n+2p)} \log n \int_{\bar\a_n}^{\log n}\frac{1}{1+2\a+2p} \,d\a
+ \log^3 n\int_{\log n}^\infty \frac{1}{(1+2\a + 2p)^3}\,d\a.
\]
 This quantity is bounded by a constant times
\[
\frac{n^{1/(1+2\overline{\a}_n+2p)} (\log n)^2}{1+2\overline{\a}_n+2p}.
\]

\subsection{$\MM_n(\a)$ on $\a\in[\overline{\a}_n-1/\log n,\overline{\a}_n]$}\label{sec: IntervalA}

In this section we show that the process $\MM_n(\a)$ is with probability going to one smaller than a negative, arbitrary large constant times $n^{1/(1+2\overline\a_n+2p)}(\log n)^3/(1+2\overline{\a}_n+2p)$ uniformly on the interval $[\overline{\a}_n-1/\log n,\overline{\a}_n]$. More precisely, we show that for every $\b, R, M > 0$, the constant $L > 0$ in the definition of $\bar\a_n$
can be chosen such that
\begin{align}
\label{eq: BoundSupSupM}
\limsup_{n\rightarrow\infty}\sup_{\|\mu_0\|_\b \le R}\sup_{\a\in[\overline{\a}_n-1/\log n,\overline{\a}_n]}\E_0\frac{(1+2\a+2p)\MM_n(\a)}{n^{1/(1+2\a+2p)}(\log n)^3}<-M\\
\label{eq: BoundSupM}
\sup_{\|\mu_0\|_\b \le R} \E_0 \sup_{\a\in[\overline{\a}_n-1/\log n,\overline{\a}_n]}\frac{(1+2\a+2p)|\MM_n(\a)-\E_0 \MM_n(\a)|}{n^{1/(1+2\a+2p)}(\log n)^3}\rightarrow 0.
\end{align}

The expected value of the normalized version of the process $\MM_n$ given on the left-hand side of \eqref{eq: BoundSupSupM} is equal to
\begin{equation}
\label{eq: EMn}
\frac{1+2\a+2p}{n^{1/(1+2\a+2p)}(\log n)^3}\Big(\sum_{i=1}^{\infty}\frac{n^2\log i}{(i^{1+2\a+2p}+n)^2}
-\sum_{i=1}^{\infty}\frac{n^2i^{1+2\a}\mu_{0,i}^2\log i}{(i^{1+2\a+2p}+n)^2}\Big).
\end{equation}
We write this as the sum of two terms and bound the first term by
\begin{align*}
\frac{1+2\a+2p}{n^{1/(1+2\a+2p)}(\log n)^3}\sum_{i=1}^{\infty}\frac{n\log i}{i^{1+2\a+2p}+n}.
\end{align*}
We want to bound this quantity for $\a\in[\overline{\a}_n-1/\log n,\overline{\a}_n]$.
By Lemma \ref{lem: abar},  $\b/4 < \overline{\a}_n-1/\log n$ for large enough $n$,
so this interval is included in $(\b/4, \infty)$.
Taking $c=\beta/2+2p$ in
Lemma~\ref{lem: LogSeries}.(i) then shows that the first term is bounded
by a multiple of $1/(\log n)^2$ and hence
tends to zero, uniformly over $[\overline\a_n-1/\log n,\overline\a_n]$.
We now consider the second term in $\eqref{eq: EMn}$, which is equal to $h_n(\a)/(\log n)^2$. By Lemma~\ref{lem: hbounds} for any $\mu_0 \in \ell_2$ and $n \geq e^4$ we have
\[
\frac{h_n(\a)}{(\log n)^2} \gtrsim \frac{1}{(\log n)^2}h_n(\overline{\a}_n) = L,
\]
where the last equality holds by the definition of $\overline{\a}_n$.
This concludes the proof of \eqref{eq: BoundSupSupM}.

To verify \eqref{eq: BoundSupM} it suffices, by Corollary 2.2.5 in \cite{vdVW} (applied with $\psi(x) = x^2$), to show that
\begin{equation}\label{eq: BoundVar}
\sup_{\|\mu_0\|_\beta \le R} \sup_{\a\in[\overline{\a}_n-1/\log n,\overline{\a}_n]}\var_0 \frac{(1+2\a+2p)\MM_n (\a)}{n^{1/(1+2\a+2p)}(\log n)^3} \to 0,
\end{equation}
and
\[
\sup_{\|\mu_0\|_\beta \le R} \int_{0}^{\diam_n}\sqrt{N(\e, [\overline{\a}_n-1/\log n, \overline{\a}_n], d_n)}\, d\e \to 0,
\]
where $d_n$ is the semimetric defined by
\[
d_n^2(\a_1, \a_2) = \var_0 \Bigl(\frac{(1+2\a_1+2p)\MM_n (\a_1)}{n^{1/(1+2\a_1+2p)}(\log n)^3} - \frac{(1+2\a_2+2p)\MM_n (\a_2)}{n^{1/(1+2\a_2+2p)}(\log n)^3}\Bigr),
\]
$\diam_n$ is the diameter of $[\overline{\a}_n-1/\log n, \overline{\a}_n]$ relative do $d_n$, and $N(\e, B, d)$ is the minimal number of $d$-balls of radius $\e$ needed to cover the set $B$.

By Lemma~\ref{lem: VarLemma1}
\begin{equation}\label{eq: Var}
\var_0 \frac{(1+2\a+2p)\MM_n (\a)}{n^{1/(1+2\a+2p)}(\log n)^3} \lesssim \frac{n^{-1/(1+2\a+2p)}}{(\log n)^4} \bigl(1+h_n(\a)\bigr),
\end{equation}
(with an implicit constant that does not depend on $\mu_0$ and $\a$). By the definition of $\overline{\a}_n$ the function $h_n(\a)$ is bounded above by $L(\log n)^2$ on the interval $[\overline{\a}_n-1/\log n, \overline{\a}_n]$. Together with \eqref{eq: Var} it proves \eqref{eq: BoundVar}.

The last bound also shows that the $d_n$-diameter of the set $[\overline{\a}_n-1/\log n, \overline{\a}_n]$ is bounded above by a
constant times $(\log n)^{-1}$, with a constant that does not depend on $\mu_0$ and $\a$.
By Lemma~\ref{lem: VarDist} and the fact that $h_n(\a) \le L(\log n)^2$ for  $\a \in [\overline{\a}_n-1/\log n, \overline{\a}_n]$, we get the
upper bound, $\a_1, \a_2 \in  [\overline{\a}_n-1/\log n, \overline{\a}_n]$,
\[
d_n(\a_1, \a_2) \lesssim |\a_1 - \a_2|,
\]
with a constant that does not depend on $\mu_0$.
Therefore
$N(\e, [\overline{\a}_n-1/\log n, \overline{\a}_n], d_n) \lesssim 1/(\e\log n)$ and hence
\[
\sup_{\|\mu_0\|_\beta  \le R} \int_{0}^{\diam_n}\sqrt{N(\e, [\overline{\a}_n-1/\log n, \overline{\a}_n], d_n)}\, d\e \lesssim
\frac{1}{\log n} \to 0.
\]

\subsection{$\MM_n(\a)$ on $(0, \underline{\a}_n+1/\log n]$}\label{sec: UnderA}

In this subsection we prove that if the constant $l$ in the definition of $\underline{\a}_n$ is small enough, then
\begin{align}
\label{eq: BoundInfM}
\liminf_{n\to\infty}\inf_{\mu_0\in\ell_2}\inf_{\a\in(0,\underline{\a}_n+1/\log n]}\E_0\frac{(1+2\a+2p)\MM_n(\a)}{n^{1/(1+2\a+2p)}\log n}>0\\
\label{eq: BoundSupMb}
\sup_{\mu_0\in\ell_2}\E_0\sup_{\a\in(0,\underline{\a}_n+1/\log n]}\frac{(1+2\a+2p)|\MM_n(\a)-\E_0 \MM_n(\a)|}{n^{1/(1+2\a+2p)}\log n}\to 0.
\end{align}
This shows that $\MM_n$ is positive throughout $(0, \underline{\a}_n+1/\log n]$ with probability tending to one uniformly over $\ell_2$.

Since $\E_0 Y_i^2 = \k_i^2\mu_{0,i}^2+1/n$, the expected value on the left-hand side of \eqref{eq: BoundInfM} is equal to
\begin{equation}\label{eq: EMnCase2}
\frac{1+2\a+2p}{n^{1/(1+2\a+2p)}\log n}\sum_{i=1}^{\infty}\frac{n^2\log i}{(i^{1+2\a+2p}+n)^2}
-h_n(\a).
\end{equation}
We first find a lower bound for the first term.
Since $\underline\a_n \le \sqrt{\log n}$ by definition, we have $\a \ll \log n$ for all
$\a \in (0, \underline{\a}_n+1/\log n]$. Then
it follows from Lemma~\ref{lem: LogLower} that for $n$ large enough,
the first term in \eqref{eq: EMnCase2} is bounded from below by $1/12$
for all
$\a \in (0, \underline{\a}_n+1/\log n]$.
Next note that by definition of $h_n$ and Lemma~\ref{lem: hbounds}, we have
\[
\sup_{\a\in(0,\underline{\a}_n+1/\log n]}h_n(\a) \le K l,
\]
where $K> 0$ is a constant independent of $\mu_0$.
So by choosing $l > 0$ small enough, we can indeed ensure that \eqref{eq: BoundInfM} is true.

To verify \eqref{eq: BoundSupMb} it suffices again, by Corollary 2.2.5 in \cite{vdVW} applied with $\psi(x) = x^2$, to show that
\begin{equation}\label{eq: BoundVarb}
\sup_{\mu_0\in\ell_2} \sup_{\a\in(0,\overline{\a}_n+1/\log n]}\var_0 \frac{(1+2\a+2p)\MM_n (\a)}{n^{1/(1+2\a+2p)}\log n} \to 0,
\end{equation}
and
\[
\sup_{\mu_0\in\ell_2} \int_{0}^{\diam_n}\sqrt{N(\e, (0, \underline{\a}_n+1/\log n], d_n)}\, d\e \to 0,
\]
where $d_n$ is the semimetric defined by
\[
d_n^2(\a_1, \a_2) = \var_0 \Bigl(\frac{(1+2\a_1+2p)\MM_n (\a_1)}{n^{1/(1+2\a_1+2p)}\log n} - \frac{(1+2\a_2+2p)\MM_n (\a_2)}{n^{1/(1+2\a_2+2p)}\log n}\Bigr),
\]
$\diam_n$ is the diameter of $(0, \underline{\a}_n+1/\log n]$ relative to $d_n$, and $N(\e, B, d)$ is the minimal number of $d$-balls of radius $\e$ needed to cover the set $B$.

By Lemma~\ref{lem: VarLemma1}
\begin{equation}\label{eq: Varb}
\var_0 \frac{(1+2\a+2p)\MM_n (\a)}{n^{1/(1+2\a+2p)}\log n} \lesssim n^{-1/(1+2\a+2p)}\bigl(1+h_n(\a)\bigr),
\end{equation}
with a constant that does not depend on $\mu_0$ and $\a$.
We have seen that on the interval $(0, \underline{\a}_n+1/\log n]$ the function $h_n$ is bounded by a
constant times $l$, hence the variance in \eqref{eq: BoundVarb} is bounded by a multiple of $n^{-1/(1+2\underline{\a}_n+2/\log n+2p)} \leq e^{-(1/3)\sqrt{\log n}}\to 0$, which proves \eqref{eq: BoundVarb}.

The variance bound above also imply that the $d_n$-diameter of the set $(0, \underline{\a}_n+1/\log n]$ is bounded by a multiple of
$e^{-(1/6)\sqrt{\log n}}$.
%
By Lemma~\ref{lem: VarDist}, the definition of $\underline\a_n$ and Lemma \ref{lem: hbounds},
\begin{align*}
d_n(\a_1, \a_2) \lesssim |\a_1-\a_2|(\log n)\sqrt{n^{-1/(1+2\underline{\a}_n+2/\log n+2p)}} \lesssim |\a_1-\a_2|,
\end{align*}
with constants that do not depend on $\mu_0$. Hence for the covering number of $(0, \underline{\a}_n+1/\log n]\subset(0, 2\sqrt{\log n})$ we have
\[
N(\e, (0, \underline{\a}_n+1/\log n], d_n) \lesssim \frac{\sqrt{\log n}}{\e},
\]
and therefore
\begin{align*}
\sup_{\mu_0\in\ell_2} \int_{0}^{\diam_n}\sqrt{N(\e, (0, \underline{\a}_n+1/\log n], d_n)}\, d\e
&\lesssim (\log n)^{1/4}e^{-(1/12)\sqrt{\log n}} \to 0.
\end{align*}

\subsection{Bounds on $h_n(\a)$, variances and distances}

In this section we prove a number of auxiliary lemmas used in the preceding.
The first one is about the behavior of the function $h_n$ in a neighborhood of $\underline\a_n$ and $\bar\a_n$.

\begin{lemma}\label{lem: hbounds}
The function $h_n$ satisfies the following bounds:
\begin{align*}
h_n(\a)\gtrsim h_n(\overline\a_n), &\quad \text{ for } \a \in \Bigl[\overline{\a}_n-\frac{1}{\log n},\overline{\a}_n\Bigr] \text{  and  } n \geq e^4,
\\
h_n(\a)\lesssim h_n(\underline\a_n), &\quad \text{ for } \a \in \Bigl[\underline{\a}_n,\underline{\a}_n+\frac{1}{\log n}\Bigr] \text{  and  } n\geq e^2.
\end{align*}
\end{lemma}

\begin{proof}
We provide a detailed proof of the first inequality, the second one can be proved using similar arguments.

 Let
\[
S_n(\a) = \sum_{i=1}^{\infty}\frac{n^2i^{1+2\a} \mu_{0,i}^2\log i}{(i^{1+2\a + 2p}+n)^2}
\]
be the sum in the definition of $h_n$.
Splitting the sum into two parts we get, for $\a \in [\bar \a_n - 1/\log n , \bar\a_n]$,
\begin{align*}
4 S_n(\a)
&\geq \sum_{i\le n^{1/(1+2\a+2p)}}i^{1+2\overline\a_n-2/\log n}\mu_{0,i}^2\log i\\
& \quad +{n^2}\sum_{i > n^{1/(1+2\a+2p)}}i^{-1-2\overline\a_n-4p}\mu_{0,i}^2\log i.
\end{align*}
In the first sum $i^{-2/\log n}$ can be bounded below by $\exp(-2)$.  Furthermore, for $i\in[n^{1/(1+2\overline\a_n+2p)},n^{1/(1+2\a+2p)}]$, we have the  inequality
\[
i^{1+2\overline\a_n} \mu_{0,i}^2\log i\geq n^2 i^{-1-2\overline\a_n-4p} \mu_{0,i}^2\log i.
\]
Therefore  $S_n(\a)$ can be bounded from below by a constant times
\begin{align*}
&\sum_{i\le n^{1/(1+2\overline\a_n+2p)}}i^{1+2\overline\a_n}\mu_{0,i}^2\log i
+{n^2}\sum_{i>n^{1/(1+2\overline\a_n+2p)}}i^{-1-2\overline\a_n-4p}\mu_{0,i}^2\log i\\
&\qquad\geq  \sum_{i\le n^{1/(1+2\overline\a_n+2p)}}\frac{n^2i^{1+2\overline\a_n}\mu_{0,i}^2\log i}{(i^{1+2\overline{\a}_n+2p}+n)^2}
+\sum_{i>n^{1/(1+2\overline\a_n+2p)}}\frac{n^2i^{1+2\overline\a_n}\mu_{0,i}^2\log i}{(i^{1+2\overline{\a}_n+2p}+n)^2}.
\end{align*}
Hence, we have $S_n(\a) \gtrsim S_n(\overline{\a}_n)$
for $\a \in [\overline{\a}_n-1/\log n,\overline{\a}_n]$.

Next note that for $n \geq e^4$ we have $2(1+2\overline{\a}_n-2/\log n+2p) \geq 1+2\overline{\a}_n+2p$. Moreover,
$n^{-{1}/({1+2\overline{\a}_n-2/\log n+2p})}\gtrsim n^{-{1}/({1+2\overline{\a}_n+2p})}$.
Therefore
\[
\frac{1+2\a+2p}{n^{1/(1+2\a+2p)}\log n} \gtrsim \frac{1+2\overline{\a}_n+2p}{n^{1/(1+2\overline{\a}_n+2p)}\log n}
\]
for $\a \in [\overline{\a}_n-1/\log n,\overline{\a}_n]$ and for $n \geq e^4$. Combining this with
the inequality for $S_n(\alpha)$ yields the desired result.
\end{proof}

Next we present two results on  variances involving the random function  $\MM_n$.

\begin{lemma}\label{lem: VarLemma1}
For any $\a > 0$,
\[
\var_0 \frac{(1+2\a+2p)\MM_n (\a)}{n^{1/(1+2\a+2p)}} \lesssim n^{-1/(1+2\a+2p)}(\log n)^2 \bigl(1+h_n(\a)\bigr).
\]
\end{lemma}

\begin{proof}
The random variables $Y_i^2$ are independent and $\var_0 Y_i^2 = 2/n^2+4\k_i^2\mu_{0,i}^2/n$, hence the
variance in the statement of the lemma is equal to
\begin{equation}\label{eq: VarDev}
\begin{split}
&\frac{2n^2(1+2\a+2p)^2}{n^{2/(1+2\a+2p)}}\sum_{i=1}^\infty \frac{i^{2+4\a + 4p}(\log i)^2}{(i^{1+2\a+2p}+n)^4}\\
& \quad + \frac{4n^3(1+2\a+2p)^2}{n^{2/(1+2\a+2p)}}\sum_{i=1}^\infty \frac{i^{2+4\a+2p}(\log i)^2\mu_{0,i}^2}{(i^{1+2\a+2p}+n)^4}.
\end{split}
\end{equation}
 By Lemma~\ref{lem: Botond'sTrick}  the first term is bounded by
\begin{align*}
&\frac{2n(1+2\a+2p)\log n }{n^{2/(1+2\a+2p)}}\sum_{i=1}^\infty  \frac{i^{1+2\a+2p}\log i}{(i^{1+2\a+2p}+n)^2}\\
& \quad \le \frac{2(1+2\a+2p)\log n }{n^{2/(1+2\a+2p)}}\sum_{i=1}^\infty  \frac{n\log i}{i^{1+2\a+2p}+n}.
\end{align*}
Lemma~\ref{lem: LogSeries}.(i) further bounds the right hand side of the above display by a multiple of $n^{-1/(1+2\a+2p)}(\log n)^2$ uniformly for  $\a > c$, where $c > 0$ is an arbitrary constant.
For  $\a \leq c$ we get the same bound by applying Lemma~\ref{lem: LogM} (with $m = 2$, $l=4$, $r=1+2\a+2p$, $r_0 = 1+2c+2p$, and $s=2r$) to the first term in \eqref{eq: VarDev}.
 By  Lemma~\ref{lem: Botond'sTrick}, the  second term in \eqref{eq: VarDev} is bounded by
\begin{align*}
& 4n^{-2/(1+2\a+2p)}(1+2\a+2p)(\log n)\sum_{i=1}^\infty  \frac{n^2i^{1+2\a}\mu_{0,i}^2\log i}{(i^{1+2\a}\k_i^{-2}+n)^2}\\
 & \quad = 4n^{-1/(1+2\a+2p)}(\log n)^2h_n(\a).
\end{align*}
Combining the  upper bounds for the two terms we arrive at the assertion of the lemma.
\end{proof}

\begin{lemma}\label{lem: VarDist}
For any $0 < \a_1 < \a_2 <\infty$ we have that
\begin{align*}
\var_0 & \Big( \frac{(1+2\a_1+2p)\MM_n(\a_1)}{n^{1/(1+2\a_1+2p)}} - \frac{(1+2\a_2+2p)\MM_n(\a_2)}{n^{1/(1+2\a_2+2p)}} \Big)\\
&\lesssim(\a_1-\a_2)^2(\log n)^4 \sup_{\a\in[\a_1,\a_2]}n^{-1/(1+2\a+2p)}\bigl(1+h_n(\a)\bigr),
\end{align*}
with a constant that does not depend on $\a$ and $\mu_0$.
\end{lemma}

\begin{proof}
The variance we have to bound can be written as
\[
n^4\sum_{i=1}^\infty (f_i(\a_1) - f_i(\a_2))^2(\log i)^2 \var_0 Y_i^2,
\]
where
 $f_i(\a) = (1+2\a+2p)i^{1+2\a+2p}n^{-1/(1+2\a+2p)}(i^{1+2\a+2p}+n)^{-2}$.
 For the derivative of $f_i$ we have $f'_1(\a) = 2f_1(\a)(1/(1+2\a + 2p) + \log n / (1+2\a + 2p)^2)$
 and for $i \ge 2$,
 \begin{align*}
|f_i'(\a)| &= \Bigl|2f_i(\a)\Bigl(\frac{1}{1+2\a+2p} + \log i + \frac{\log n}{(1+2\a+2p)^2}-\frac{2i^{1+2\a+2p}\log i}{i^{1+2\a+2p}+n}\Bigr)\Bigr|\\
&\leq 8f_i(\a)\bigl(\log i +(\log n)/(1+2\a + 2p)^2\bigr).
\end{align*}
It follows that  the variance is bounded by a constant times
\begin{align*}
& (\a_1-\a_2)^2n^4 \sup_{\a\in[\a_1, \a_2]}(1+2\a+2p)^2\Big(\\
& \quad\sum_{i=1}^\infty \frac{i^{2+4\a+4p}(\log i)^2
\bigl(1 \vee \log i +(\log n)/(1+2\a + 2p)^2\bigr)^2}{n^{2/(1+2\a+2p)}(i^{1+2\a+2p}+n)^4}\var_0 Y_i^2\Big).
\end{align*}
Since $\var_0 Y_i^2 = 2/n^2+4\k_i^2\mu_{0,i}^2/n$,  it suffices to show that both
\begin{equation}\label{eq: BoundM}
\begin{split}
& n^2\sup_{\a\in[\a_1, \a_2]}(1+2\a+2p)^2\Big(\\
& \quad \sum_{i=1}^\infty \frac{i^{2+4\a+4p}(\log i)^2\bigl(1 \vee \log i +(\log n)/(1+2\a + 2p)^2\bigr)^2}{n^{2/(1+2\a+2p)}(i^{1+2\a+2p}+n)^4}
\Big)
\end{split}
\end{equation}
and
\begin{equation}\label{eq: BoundHb}
\begin{split}
&n^3\sup_{\a\in[\a_1, \a_2]}(1+2\a+2p)^2\Big(\\
&\quad \sum_{i=1}^\infty \frac{i^{2+4\a+2p}(\log i)^2\mu_{0,i}^2\bigl(1 \vee\log i +(\log n)/(1+2\a + 2p)^2\bigr)^2}{n^{2/(1+2\a+2p)}(i^{1+2\a+2p}+n)^4}
\Big)
\end{split}
\end{equation}
are bounded by a constant times $(\log n)^4 \sup_{\a\in[\a_1,\a_2]}n^{-1/(1+2\a+2p)}(1+h_n(\a))$.

By applying Lemma~\ref{lem: Botond'sTrick} twice (once the first statement with $r=1+2\a+2p$ and $m=1$ and
once the second one with the same $r$ and $m=3$ and $\xi=1$) the expression in \eqref{eq: BoundHb} is seen to be bounded above by
a constant times
\begin{align*}
(\log n)^3 \sup_{\a\in[\a_1,\a_2]}\Bigl(n^{-2/(1+2\a+2p)}(1+2\a+2p)
\sum_{i=1}^\infty \frac{n^2i^{1+2\a}\mu_{0,i}^2\log i}{(i^{1+2\a+2p}+n)^2}\Bigr).
\end{align*}
The expression in the parentheses equals $h_n(\a)n^{-1/(1+2\a+2p)}\log n$.
Now fix $c > 0$.
Again, applying Lemma~\ref{lem: Botond'sTrick} twice
implies that  we get that \eqref{eq: BoundM} is bounded above by
\begin{align*}
(\log n)^3 \sup_{\a\in[\a_1,\a_2]}\Bigl(\frac{2n^{-2/(1+2\a+2p)}}{1+2\a+2p}\sum_{i=1}^\infty \frac{ni^{1+2\a+2p}\log i}{(i^{1+2\a+2p}+n)^2}\Bigr).
\end{align*}
Using the inequality $x/(x+y)\leq 1$ and Lemma~\ref{lem: LogSeries}.(i), the expression in the parenthesis can be bounded by a constant
times $n^{-1/(1+2\a+2p)}\log n$ for $\a > c$.
For  $\a \le c$, Lemma~\ref{lem: LogM} (with $m = 2$ or $m=4$, $l=4$, $r=1+2\a+2p$, $r_0 = 1+2c+2p$, and $s=2r$)  gives the
same bound (or even a better one) for \eqref{eq: BoundM}.
The proof is completed by combining the obtained bounds.
\end{proof}

\section{Proof of Theorem~\ref{thm: ConvergenceEB}}\label{sec: ProofSob}

As before we only present the details of the proof for the Sobolev case $\mu_0 \in S^\beta$. The analytic case can be dealt with
similarly. Again, we assume the exact equality $\kappa_i = i^{-p}$ for simplicity.

 By Markov's inequality and Theorem \ref{thm: AlphaMagnitude},
\begin{equation}\label{eq: ll}
\begin{split}
& \sup_{\|\mu_0\|_\beta \le R} \E_0 \Pi_{\hat{\a}_n}\bigl(\|\mu-\mu_0\| \geq M_n\e_{n}\, \big|\, Y\bigr)\\
& \qquad \le \frac{1}{M_n^2\e_n^2} \sup_{\|\mu_0\|_\beta \le R} \E_0 \sup_{\a\in[\underline\a_n, \overline\a_n\wedge \log n]}
R_n(\a) + o(1),
\end{split}
\end{equation}
where
\[
R_n(\a) = \int \|\mu-\mu_0\|^2\,\Pi_\a(d\mu \given Y)
\]
is the posterior risk. We will show in the subsequent subsections that for $\e_n
=n^{-\b/(1+2\b+2p)}(\log n )^{2}(\log\log n)^{1/2}$ and arbitrary $M_n \to \infty$, the first term on the right
of \eqref{eq: ll} vanishes
as $n \to \infty$.
Note that by the explicit posterior computation \eqref{eq: PostDist}, we have
\begin{equation}
\label{eq: PostExp}
R_n(\a)
=\sum_{i=1}^\infty (\hat\mu_{\a,i}-\mu_{0,i})^2
+\sum_{i=1}^{\infty}\frac{i^{2p}}{i^{1+2\a+2p}+n},
\end{equation}
where $\hat\mu_{\a,i}=ni^{p}(i^{1+2\a+2p}+n)^{-1}Y_i$
is the $i$th coefficient of the posterior mean. We divide the Sobolev-ball $\|\mu_0\|_\beta \le R$ into two subsets
\begin{align*}
P_n  & = \{\mu_0: \|\mu_0 \|_\beta \le R, \ \overline{\a}_n \leq (\log n)/\log 2-1/2-p\},\\
Q_n  & = \{\mu_0: \|\mu_0 \|_\beta \le R, \ \overline{\a}_n > (\log n)/\log 2-1/2-p\},
\end{align*}
and show that on both subsets the posterior risks are of the order $\e_n^2$.
%
%
%

\subsection{Bound for the expected posterior risk over $P_n$}
\label{sec: 61}

In this section we  prove that
\begin{equation}\label{eq: PnExp}
\sup_{\mu_0\in P_n} \sup_{\a\in[\underline\a_n, \overline\a_n]} \E_0 R_n(\a) = O(\e_n^2).
\end{equation}
The second term of \eqref{eq: PostExp} is deterministic. The expectation of the first term can be split into square bias and variance terms. We find that the expectation of \eqref{eq: PostExp} is given by
\begin{equation}\label{eq: PostRisk}
\sum_{i=1}^\infty \frac{i^{2+4\a+4p}\mu_{0,i}^2}{(i^{1+2\a+2p}+n)^2}
+{n}\sum_{i=1}^\infty \frac{i^{2p}}{(i^{1+2\a+2p}+n)^2}
+\sum_{i=1}^\infty \frac{i^{2p}}{i^{1+2\a+2p}+n}.
\end{equation}
Note that the second and third terms in
\eqref{eq: PostRisk} are independent of $\mu_0$, and that the second is bounded by
the third. By Lemma~\ref{lem: LogM} (with $m=0$, $l=1$, $r=1+2\a+2p$ and $s=2p$)
the latter is for  $\a \ge\underline{\a}_n$ further bounded by
\begin{align*}
 n^{-\frac{2\a}{1+2\a+2p}}\leq  n^{-\frac{2\underline{\a}_n}{1+2\underline{\a}_n+2p}}.
\end{align*}
In view of Lemma \ref{lem: abar}.(i),  the right-hand side is bounded
by a constant times $n^{-2\beta/(1+2\beta+ 2p)}$ for large $n$.
%
%
%
%

It remains to consider the first sum in \eqref{eq: PostRisk}, which
we divide into three parts and show that each of the parts has the stated order. First we note that
\begin{equation}\label{eq: BetaTail}
\sum_{i > n^{1/(1+2\b+2p)}} \frac{i^{2+4\a+4p}\mu_{0,i}^2}{(i^{1+2\a+2p}+n)^2}
\le \sum_{i > n^{1/(1+2\b+2p)}}\mu_{0,i}^2 \leq \|\mu_0\|^2_\b n^{-2\b/(1+2\b+2p)}.
\end{equation}
Next, observe that elementary calculus shows that for $\a > 0$ and $n \ge e$, the maximum of the function $i \mapsto
i^{1+2\a + 4p}/\log i$ over the interval $[2, n^{1/(1+2\a + 2p)}]$ is attained
at $i = n^{1/(1+2\a + 2p)}$, for $\a\leq\log n/(2\log 2)-1/2-p$.
It follows that for $\a > 0$,
\begin{align*}
& \sum_{i \le n^{1/(1+2\a+2p)}} \frac{i^{2+4\a+4p}\mu_{0,i}^2}{(i^{1+2\a+2p}+n)^2} \\
& \quad = \frac{\mu_{0,1}^2}{(1+n)^2}
+ \frac1{n^2}\sum_{2 \le i \le n^{1/(1+2\a+2p)}} \frac{((i^{1+2\a+4p})/\log i)n^2i^{1+2\a}\mu_{0,i}^2\log i}{(i^{1+2\a+2p}+n)^2} \\
& \quad \le \frac{\mu_{0,1}^2}{(1+n)^2} + n^{-\frac{2\a}{1+2\a + 2p}}h_n(\a).
\end{align*}
We note that for $\a>\log n/(2\log 2)-1/2-p$ the second term on the right hand side of the preceding display disappears and for $\mu_0 \in P_n$ we have that $\overline\a_n$ is finite. Since $n^{1/(1+2\overline\a_n+2p)} \le n^{1/(1+2\a+2p)}$
for $\a \le \overline\a_n$, the preceding implies that
\begin{align*}
 \sup_{\mu_0\in P_n} \sup_{\a\in[\underline\a_n, \overline\a_n]} \sum_{i \le n^{1/(1+2\bar\a_n+2p)}} \frac{i^{2+4\a+4p}\mu_{0,i}^2}{(i^{1+2\a+2p}+n)^2} \quad \lle \frac{R^2}{n^2} + L n^{-\frac{2\underline\a_n}{1+2\underline\a_n + 2p}} \log^2n.
\end{align*}
By Lemma \ref{lem: abar}, $\underline\a_n \ge \beta - c_0/\log n$ for a constant $c_0 > 0$
(only depending on $\beta, R, p$). Hence,
using  that $x \mapsto x/(c+x)$ is increasing for every $c > 0$ the right-hand side is bounded
by a constant times $n^{-2\beta/(1+2\beta+ 2p)}\log^2 n$.

To complete the proof
 we deal with the terms between $n^{1/(1+2\overline{\a}_n+2p)}$ and $n^{1/(1+2\b+2p)}$.
Let $J=J(n)$ be the smallest integer such that $(\overline{\a}_n
/(1+1/\log n)^J \leq \b$. One can see that $J$ is bounded above by a multiple of $(\log n)(\log\log n)$ for any positive $\b$.
We partition the summation range under consideration into $J$ pieces using the auxiliary numbers
\[
b_j = 1+ 2\frac{\overline{\a}_n
}{(1+1/\log n)^j}+2p, \qquad j =0, \ldots, J.
\]
Note that the sequence $b_j$ is decreasing. Now we have
\[
\sum_{i= n^{1/(1+2\overline{\a}_n+2p)}}^{n^{1/(1+2\b+2p)}}
\frac{i^{2+4\a+4p}\mu_{0,i}^2}{(i^{1+2\a + 2p}+n)^2}\leq \sum_{j=0}^{J-1}\sum_{i=n^{1/b_j}}^{n^{1/b_{j+1}}}\mu_{0,i}^2 \leq 4\sum_{j=0}^{J-1}\sum_{i=n^{1/b_j}}^{n^{1/b_{j+1}}}\frac{ni^{b_j}\mu_{0,i}^2}{(i^{b_{j+1}}+n)^2},
\]
and  the  upper bound is uniform in $\a$.
Since $(b_j-b_{j+1})\log n = b_{j+1}-1-2p$, it holds for $n^{1/b_j}\leq i \leq n^{1/b_{j+1}}$ that
 $i^{b_j-b_{j+1}} \leq n^{1/\log n} = e$. On the same interval $i^{2p}$ is bounded by $n^{2p/b_{j+1}}$. Therefore the right hand side of the preceding display is further bounded by a constant times
\begin{align*}
&\sum_{j=0}^{J-1}\sum_{i=n^{1/b_j}}^{n^{1/b_{j+1}}}\frac{ni^{b_{j+1}}\mu_{0,i}^2\log i}{(i^{b_{j+1}}+n)^2}\leq \sum_{j=0}^{J-1}n^{2p/b_{j+1}-1}\sum_{i=n^{1/b_j}}^{n^{1/b_{j+1}}}\frac{n^2i^{b_{j+1}-2p}\mu_{0,i}^2\log i}{(i^{b_{j+1}}+n)^2}\\
&\leq \sum_{j=0}^{J-1}n^{2p/b_{j+1}-1}h_n\biggl(\frac{\overline{\a}_n
}{(1+1/\log n)^{j+1}}\biggr)n^{1/b_{j+1}}\frac{\log n}{b_{j+1}}\\
&\leq (\log n)\sum_{j=0}^{J-1}n^{(1+2p-b_{j+1})/b_{j+1}}h_n(b_{j+1}/2-1/2-p)\\
&\leq  (\log n) n^{-\frac{2\b/(1+1/\log n)}{1+2\b/(1+1/\log n)+2p}}\sum_{j=0}^{J-1}h_n(b_{j+1}/2-1/2-p).
\end{align*}
In the last step we used the fact that by construction, $b_{j}/2-1/2-p \geq \b/(1+1/\log n)$ for
$j\le J$. Because $b_{j}/2-1/2-p \le\overline\a_n$ for every $j\ge0$, it follows from
the definition of $\overline{\a}_n$ that $h_n(b_{j}/2-1/2-p)$
is bounded above by $L(\log n)^2$, and we recall that $J=J(n)$ is bounded
above by a multiple of $(\log n)(\log\log n)$. Finally we note that
\begin{align*}
n^{-\frac{2\b/(1+1/\log n)}{1+2\b/(1+1/\log n)+2p}}  \le en^{-2\b/(1+2\b+2p)}.
\end{align*}
Therefore the first sum in \eqref{eq: PostRisk} over the range $[n^{1/(1+2\overline{\a}_n+2p)},
n^{1/(1+2\b+2p)}]$ is bounded above by a multiple of $n^{-2\b/(1+2\b+2p)}(\log n)^4(\log\log n)$, in
the appropriate uniform sense over $P_n$. Putting the bounds above together we conclude $\eqref{eq: PnExp}$

\subsection{Bound for the centered posterior risk over $P_n$}\label{sec: brisk}

We show in this section that for the set $P_n$ we also have
\[
\sup_{\mu_0\in P_n}\E_0 \sup_{\a\in[\underline{\a}_n,\overline{\a}_n]}\Bigl|\sum_{i=1}^\infty \bigl(\hat\mu_{\a,i}-\mu_{0,i}\bigr)^2 - \E_0 \sum_{i=1}^\infty \bigl(\hat\mu_{\a,i}-\mu_{0,i}\bigr)^2\Bigr| = O(\e_n^2),
\]
for $\e_n = n^{-\b/(1+2\b+2p)}(\log n )^{2}(\log\log n)^{1/2}$.
Using the explicit expression for the posterior mean $\hat\mu_{\a,i}$ we see that the random variable in the supremum is the absolute value of $\VV(\a)/n-2\WW(\a)/\sqrt{n}$, where
\[
\VV(\a)=\sum_{i=1}^\infty \frac{n^2\k_i^{-2}}{(i^{1+2\a}\k_i^{-2}+n)^2}(Z_i^2-1), \qquad \WW(\a)= \sum_{i=1}^\infty \frac{ni^{1+2\a}\k_i^{-3}\mu_{0,i}}{(i^{1+2\a}\k_i^{-2}+n)^2}Z_i.
\]
We deal with the two processes separately.

For the process $\VV$, Corollary 2.2.5 in \cite{vdVW} implies that
\[
\E_0 \sup_{\a\in[\underline{\a}_n,\infty)}|\VV(\a)| \lle
\sup_{\a\in[\underline{\a}_n,\infty)}\sqrt{\var_0\VV(\a)} +
\int_0^{\diam_n}\sqrt{N(\e, [\underline{\a}_n,\infty), d_n)}\,d\a,
\]
where $d^2_n(\a_1, \a_2) = \var_0(\VV(\a_1)-\VV(\a_2))$ and $\diam_n$ is the $d_n$-diameter of
$[\underline{\a}_n,\infty)$.
Now the variance of $\VV(\a)$ is equal to
\[
\var_0\VV(\a) =  2n^4\sum_{i=1}^\infty \frac{i^{4p}}{(i^{1+2\a+2p}+n)^4},
\]
since $\var_0 Z_i^2 = 2$. Using Lemma~\ref{lem: LogM} (with $m=0$, $l=4$, $r=1+2\a+2p$ and $s=4p$), we can conclude that the variance of $\VV(\a)$ is bounded above by a multiple of $n^{(1+4p)/(1+2\a+2p)}$.
It follows that the diameter of the interval $\diam_n \lle n^{(1+4p)/(1+2\underline{\a}_n+2p)}$.
To  compute the covering number of the interval $[\underline{\a}_n, \infty)$
we first note that for $0 < \a_1 < \a_2$,
\begin{align*}
\var_0\bigl(\VV(\a_1)& -\VV(\a_2)\bigr) = \sum_{i=2}^\infty \biggl(\frac{n^2i^{2p}}{(i^{1+2\a_1+2p}+n)^2}-
\frac{n^2i^{2p}}{(i^{1+2\a_2+2p}+n)^2}\biggr)^2\var Z_i^2\\
&\leq 2\sum_{i=2}^\infty \frac{n^4i^{4p}}{(i^{1+2\a_1+2p}+n)^4}
\leq 2n^4\sum_{i=2}^\infty i^{-4-8\a_1-4p} \lle n^42^{-8\a_1}.
\end{align*}
Hence for $\e > 0$,
a single $\e$-ball covers the whole interval $[K\log(n/\e), \infty)$ for some constant $K > 0$.
By Lemma~\ref{lem: VarVW}, the distance $d_n(\a_1, \a_2)$ is bounded above by a multiple of $|\a_1-\a_2|n^{(1+4p)/(2+4\underline{\a}_n+4p)}(\log n)$. Therefore the covering number of the interval $[\underline{\a}_n, K\log(n/\e)]$ relative to the metric $d_n$ is bounded above by a multiple of $(\log n)n^{(1+4p)/(2+4\underline{\a}_n+4p)}(\log(n/\e))/\e$.
Combining everything we see that
\begin{align*}
\E_0\sup_{\a\in[\underline{\a}_n,\infty)}|\VV(\a)|
\lesssim n^{\frac{1+4p}{2+4\underline{\a}_n+4p}}(\log n).
\end{align*}
By the fact that $x \mapsto x/(x+c)$ is increasing and Lemma \ref{lem: abar}.(i), the right-hand side divided  by $n$ is bounded by
\[
n^{-\frac{2\underline{\a}_n}{1+2\underline{\a}_n+2p}}(\log n) \lesssim n^{-2\b/(1+2\b+2p)}(\log n).
\]

It remains to deal with the process $\WW$.
The basic line of reasoning is the same as followed above for $\VV$. An essential difference however is the
derivation of a bound for the
 variance of $\WW$, of which we provide the details. The rest of the proof is left to the reader.
The variance $\WW(\a)/\sqrt{n}$
  is given by
\[
\var_0\biggl(\frac{\WW(\a)}{\sqrt{n}}\biggr) = \sum_{i=1}^\infty \frac{ni^{2+4\a+ 6p}\mu_{0,i}^2}{(i^{1+2\a+2p}+n)^4}.
\]
We show that uniformly for $\a \in [\underline\a_n, \overline\a_n]$, this variance is bounded above by a constant (which depends only on $\|\mu_0\|_\beta$) times $n^{-(1+4\b)/(1+2\b+2p)}(\log n)^2$. We note that on the set $P_n$ the upper bound $\overline\a_n\leq \log n/\log2-1/2-p$ is finite.

For the sum over  $i \leq n^{1/(1+2\a+2p)}$ we have
\begin{equation}\label{eq: piet}
\begin{split}
&\sum_{i \le n^{1/(1+2\a+2p)}} \frac{ni^{2+4\a+6p}\mu_{0,i}^2}{(i^{1+2\a+2p}+n)^4}\\
& \qquad \leq
\frac{\mu_{0,1}^2}{n^3} +
\frac{1}{n^3}\sum_{2 \le i \le n^{1/(1+2\a+2p)}} \frac{n^2i^{1+2\a+6p}(\log i)^{-1}i^{1+2\a}\mu_{0,i}^2\log i}{(i^{1+2\a+2p}+n)^2}\\
&\qquad\leq
\frac{\|\mu_0\|^2_\b}{n^3} +
(1+2\a+2p)\frac{n^{4p/(1+2\a+2p)}}{(\log n)n^2}\sum_{i \le n^{1/(1+2\a+2p)}} \frac{n^2i^{1+2\a}\mu_{0,i}^2\log i}{(i^{1+2\a+2p}+n)^2}\\
&\qquad\leq
\frac{\|\mu_0\|^2_\b}{n^3} +
n^{-\frac{1+4\a}{1+2\a+2p}}h_n(\a).
\end{split}
\end{equation}
We note that the second term on the right hand side of the preceding display disappears for $\a>\log n/(2\log2)-1/2-p$. We have used again the fact that on the range $i \leq n^{1/(1+2\a+2p)}$, the quantity $i^{1+2\a+6p}(\log i)^{-1}$
is maximal for the largest $i$.
Now the function $x \mapsto -(1+2x)/(x+c)$ is  decreasing on $(0, \infty)$
for any $c > 1/2$. Moreover $h_n(\a) \leq L(\log n)^2$ for any $\a \leq \overline{\a}_n$, thus the preceding display is bounded above by a multiple of $n^{-(1+4\underline{\a}_n)/(1+2\underline{\a}_n+2p)}(\log n)^2$. Using Lemma \ref{lem: abar}.(i) this
is further bounded by a constant times $n^{-(1+4\b)/(1+2\b+2p)}(\log n)^2$.

Next we consider sum over the range $i > n^{1/(1+2\a+2p)}$. We distinguish two cases according to the value of $\a$. First suppose that $1+2\a \geq 2p$. Then $i^{-1-2\a+2p}(\log i)^{-1}$ is decreasing in $i$, hence
\begin{align*}
& \sum_{i>n^{1/(1+2\a+2p)}}\frac{ni^{2+4\a+ 6p}\mu_{0,i}^2}{(i^{1+2\a+2p}+n)^4} \\
&\quad \leq \frac{1}{n}\sum_{i>n^{1/(1+2\a+2p)}} \frac{n^2i^{-1-2\a+2p}(\log i)^{-1}i^{1+2\a}\mu_{0,i}^2\log i}{(i^{1+2\a+2p}+n)^2}\\
&\quad \leq \frac{1+2\a+2p}{n^{(2+4\a)/(1+2\a+2p)}\log n} \sum_{i>n^{1/(1+2\a+2p)}}
\frac{n^2i^{1+2\a}\mu_{0,i}^2\log i}{(i^{1+2\a+2p}+n)^2}\\
&\quad \leq  n^{-\frac{1+4\a}{1+2\a+2p}}h_n(\a).
\end{align*}
As above, this  is further bounded by a constant times the desired rate $n^{-(1+4\b)/(1+2\b+2p)}(\log n)^2$.
If $1+2\a < 2p$, then
\begin{align*}
\sum_{i>n^{1/(1+2\a+2p)}} \frac{ni^{2+4\a+6p}\mu_{0,i}^2}{(i^{1+2\a+2p}+n)^4}
&\leq n\sum_{i>n^{1/(1+2\a+2p)}} i^{-2-4\a-2p-2\b}i^{2\b}\mu_{0,i}^2\\
&\leq \|\mu_0\|_\b^2n^{\frac{2p-2\b}{1+2\a+2p}-1}.
\end{align*}
Since $\underline{\a}_n \geq \b - c_0/\log n$, we have $1+2\a>2\b$ for large enough $n$, for any $\a \in [\underline{\a}_n,\overline{\a}_n]$. Since we have assumed $1+2\a < 2p$, this implies that $2p > 2\beta$. Therefore the right hand side of the preceding display attains its maximum at $\a=\underline{\a}_n$.
Using again that $\underline{\a}_n \geq \b - c_0/\log n$, it is straightforward to show that for
$\a \in [\underline{\a}_n,\overline{\a}_n]$,
\[
n^{\frac{2p-2\b}{1+2\a+2p}-1} \le n^{\frac{2p-2\b}{1+2\underline\a_n+2p}-1}
\le  e^{4c_0}n^{-\frac{1+4\b}{1+2\b+2p}}.
\]

\subsection{Bound for the expected and centered posterior risk over $Q_n$}\label{sec: 63}
To complete the proof of  Theorem~\ref{thm: ConvergenceEB} we show that similar results to Sections \ref{sec: 61} and \ref{sec: brisk} hold over the set $Q_n$ as well:
\begin{equation}\label{eq: Qn1}
\sup_{\mu_0\in Q_n} \sup_{\a\in[\underline\a_n, \infty)} \E_0 R_n(\a) = O(\e_n^2),
\end{equation}
\begin{equation}\label{eq: Qn2}
\sup_{\mu_0\in Q_n}\E_0 \sup_{\a\in[\underline{\a}_n,\infty)}\Bigl|\sum_{i=1}^\infty \bigl(\hat\mu_{\a,i}-\mu_{0,i}\bigr)^2 - \E_0 \sum_{i=1}^\infty \bigl(\hat\mu_{\a,i}-\mu_{0,i}\bigr)^2\Bigr| = O(\e_n^2).
\end{equation}
For the first statement $\eqref{eq: Qn1}$ we follow the same line of reasoning as in Section
\ref{sec: 61}. The second and third terms in \eqref{eq: PostRisk} are free of $\mu_0$,
and hence the same upper bound as in Section \ref{sec: 61} apply.
The first term in \eqref{eq:  PostRisk} is also treated exactly as in Section \ref{sec: 61},
except that $n^{1/(1+2\overline{\a}_n+2p)}\le 2$ if $\mu_0\in Q_n$ and hence
the sum over the terms $i< n^{1/(1+2\overline\a_n+2p}$ need not be treated,
and we can proceed by replacing $\overline\a_n$ by $\log n/(2\log 2)-1/2-p$ in the definition
of $J$ and the sequence $b_j$.

To bound the centered posterior risk $\eqref{eq: Qn2}$ we follow the proof given in Section
\ref{sec: brisk}. There the process $\VV(\a)$ is already bounded uniformly over $[\underline\a_n,\infty)$,
whence it remains to deal with the process $\WW(\a)$. The only essential difference
is the upper bound for the variance of the process $\WW(\a)/\sqrt n$. In Section \ref{sec:
  brisk} this was shown to be  bounded above by a
multiple of the desired rate $(\log n)^2n^{-(1+4\beta)/(1+2\beta+2p)}$
for $\alpha\in[\underline{\a}_n,\overline{\a}_n\wedge (\log n/\log2-1/2-p)]$,
which is $\alpha\in[\underline{\a}_n,\log n/\log2-1/2-p]$ on the set $Q_n$.
Finally, for $\alpha\geq \log n/\log 2-1/2-p$ we have
\begin{equation}\label{eq: Var2}
\begin{split}
\sum_{i=1}^{\infty}\frac{ni^{2+4\a+6p}\mu_{0,i}^2}{(i^{1+2\a+2p}+n)^4}&\leq \frac{\mu_{0,1}^2}{n^3}+\sum_{i=2}^{\infty}\frac{ni^{-1-2\a}\mu_{0,i}^2}{i^{1+2\a+2p}+n}\\
&\leq \frac{\|\mu_{0}\|_{\beta}^2}{n^3}+\sum_{i=2}^{\infty}i^{-1-2\a-2\beta}i^{2\beta}\mu_{0,i}^2\\
&\leq \frac{\|\mu_{0}\|_{\beta}^2}{n^3}+2^{-1-2\a}\|\mu_0\|_{\beta}^2\leq \frac{\|\mu_{0}\|_{\beta}^2}{n^3}+2^{2p}\frac{\|\mu_{0}\|_{\beta}^2}{n^2}\\
&\lesssim n^{-(1+4\beta)/(1+2\beta+2p)}(\log n)^2.
\end{split}
\end{equation}

This completes the proof.

\subsection{Bounds for the semimetrics associated to  $\VV$ and $\WW$}

The following lemma is used in Section \ref{sec: brisk}.

\begin{lemma}\label{lem: VarVW}
For any $\underline{\a}_n \leq \a_1 < \a_2$ the following inequalities hold:
\[
\var_0\bigl(\VV(\a_1)-\VV(\a_2)\bigr) \lle (\a_1-\a_2)^2n^{(1+4p)/(1+2\underline{\a}_n+2p)}(\log n)^2,
\]
\[
\var_0\biggl(\frac{\WW(\a_1)}{\sqrt{n}}-\frac{\WW(\a_2)}{\sqrt{n}}\biggr)\lesssim (\a_1-\a_2)^2 n^{-\frac{1+4\b}{1+2\b+2p}}
(\log n)^4,
\]
with a constant that does not depend on $\a$ and $\mu_0$.
\end{lemma}

\begin{proof}
The left-hand side of the first inequality is equal to
\[
n^4\sum_{i=1}^\infty (f_i(\a_1)-f_i(\a_2))^2i^{4p}\var Z_i^2,
\]
where  $f_i(\a) = (i^{1+2\a+2p}+n)^{-2}$.
The derivative of $f_i$ is given by  $f_i'(\a) = -4i^{1+2\a+2p}(\log i)/(i^{1+2\a+2p}+n)^{3}$, hence the preceding display is bounded above by a multiple of
\begin{align*}
& (\a_1-\a_2)^2n^4\sup_{\a\in[\a_1, \a_2]} \sum_{i=1}^\infty \frac{i^{2+4\a+8p}(\log i)^2}{(i^{1+2\a+2p}+n)^6}\\
&\leq (\a_1-\a_2)^2n^3(\log n)^2\sup_{\a\in[\a_1, \a_2]} \frac{1}{(1+2\a+2p)^2}\sum_{i=1}^\infty \frac{i^{1+2\a+6p}}{(i^{1+2\a+2p}+n)^4}\\
&\lle (\a_1-\a_2)^2(\log n)^2 \sup_{\a\in[\a_1, \a_2]}
n^{(1+4p)/(1+2\a+2p)},
\end{align*}
with the help of Lemma~\ref{lem: Botond'sTrick} (with $r=1+2\a+2p$, and $m=2$), and Lemma~\ref{lem: LogM} (with $m=0$, $l = 4$, $r=1+2\a+2p$, and $s=r+4p$). Since $\a \geq \underline{\a}_n$, we get the first assertion of the lemma.

We next consider $\WW/\sqrt{n}$. The left-hand side of the second inequality in the statement of the lemma is equal to
\[
\sum_{i=1}^\infty (f_i(\a_1) - f_i(\a_2))^2n\mu_{0,i}^2\var Z_i,
\]
where now $f_i(\a)=i^{1+2\a+3p}/(i^{1+2\a+2p}+n)^2$.
The derivative of this $f_i$ satisfies
$|f_i'(\a)| \le 2(\log i)f_i(\a)$, hence we get the upper bound
\[
4(\a_2-\a_1)^2 \sup_{\a \in [\a_1, \a_2]} \sum_{i=1}^\infty
\frac{n i^{2+4\a+6p}\mu^2_{0,i}\log^2 i}{(i^{1+2\a+2p}+n)^4}.
\]
The proof is completed by arguing as in (\ref{eq: piet}) or (\ref{eq: Var2}).
\end{proof}

\section{Proof of Theorem~\ref{thm: ConvergenceHB}}\label{sec: ProofHB}
Again we only provide details for the Sobolev case.
Let $A_n$ be the event that $\hat\a_n \in [\underline{\a}_n,\overline{\a}_n]$.
Then with $\a \mapsto \lambda_n(\a \given Y)$ denoting the posterior Lebesgue density of $\a$, we have
{\begin{equation}\label{eq: FullPost}
\begin{split}
& \sup_{\|\mu_0\|_\b \le R} \E_0 \Pi(\|\mu-\mu_0\|\geq M_nL_nn^{-\b/(1+2\b + 2p)}|Y)\\
&\quad\leq \sup_{\|\mu_0\|_\b \le R} \Pr_0(A_n^c) + \sup_{\|\mu_0\|_\b \le R} \E_0\int_0^{\underline{\a}_n}\l_n(\a|Y)\, d\a\, 1_{A_n}\\
&\quad \quad{+}\sup_{\|\mu_0\|_\b \le R}\E_0 \int_{\underline{\a}_n}^\infty \l_n(\a|Y)\Pi_\a(\|\mu-\mu_0\|\geq M_n
L_nn^{-\b/(1+2\b + 2p)}|Y)\, d\a\, 1_{A_n}.
\end{split}
\end{equation}}
By Theorem \ref{thm: AlphaMagnitude} the first term on the right vanishes as $n \to \infty$, provided $l$ and $L$ in the definitions
of $\underline\a_n$ and $\bar\a_n$ are chosen small and large enough, respectively.  We will show that the
other terms tend to $0$ as well.

Observe that
$\l_n(\a\given Y) \propto L_n(\a)\l(\a)$,
where $L_n(\a) = \exp(\ell_n(\a))$, for $\ell_n$ the random function defined by \eqref{eq: ell}.
In Section~\ref{sec: UnderA} we have shown that on the interval $(0, \underline{\a}_n+1/\log n]$
\[
\ell'_n(\a) = \MM_n(\a) \gtrsim \frac{n^{1/(1+2\a+2p)}\log n}{1+2\a+2p},
\]
 on the event $A_n$. Therefore on the interval $(0, \underline{\a}_n]$ we have
\[
\ell_n(\a) < \ell_n(\underline{\a}_n) \leq \ell_n\Bigl(\underline{\a}_n + \frac{1}{2\log n}\Bigr) - \frac{{K}n^{1/(1+2\underline{\a}_n+2p)}}{1+2\underline{\a}_n+2p}
\]
for some $K > 0$ and on the interval $[\underline{\a}_n+1/(2\log n), \underline{\a}_n +1/\log n]$,
\[
\ell_n(\a) \geq \ell_n\Bigl(\underline{\a}_n + \frac{1}{2\log n}\Bigr).
\]
For the likelihood $L_n$ we have the corresponding bounds
\[
L_n(\a) < \exp\Bigl(-\frac{Kn^{1/(1+2\underline{\a}_n+2p)}}{1+2\underline{\a}_n+2p} \Bigr)L_n\Bigl(\underline{\a}_n +\frac{1}{2\log n}\Bigr)
\]
for $\a \in (0, \underline{\a}_n]$ and
\[
L_n(\a) \geq L_n\Bigl(\underline{\a}_n+\frac{1}{2\log n}\Bigr)
\]
for $\a \in [\underline{\a}_n + 1/(2\log n), \underline{\a}_n + 1/\log n]$ on the event $A_n$.
Using these estimates for $L_n$ we obtain the following upper bound for the second term on the right-hand side of \eqref{eq: FullPost}:
\begin{equation}\label{eq: SecondTermPost}
\begin{split}
&\sup_{\|\mu_0\|_\b \le R} \E_0 \frac{\int_{0}^{\underline{\a}_n}\l(\a)L_n(\a)\, d\a}{\int_{0}^{\infty}\l(\a)L_n(\a)\, d\a}\\
&\quad \leq \sup_{\|\mu_0\|_\b \le R} \E_0\exp\Bigl(-\frac{Kn^{1/(1+2\underline{\a}_n+2p)}}{1+2\underline{\a}_n+2p} \Bigr) \frac{L_n\Bigl(\underline{\a}_n + \frac{1}{2\log n}\Bigr)\int_{0}^{\underline{\a}_n}\l(\a)\, d\a}{L_n\Bigl(\underline{\a}_n + \frac{1}{2\log n}\Bigr)\int_{\underline{\a}_n+1/(2\log n)}^{\underline{\a}_n+1/\log n}\l(\a)\, d\a}\\
&\quad \leq \sup_{\|\mu_0\|_\b \le R} \exp\Bigl(-\frac{Kn^{1/(1+2\underline{\a}_n+2p)}}{1+2\underline{\a}_n+2p} \Bigr) \Bigl(\int_{\underline{\a}_n+1/(2\log n)}^{\underline{\a}_n+1/\log n}\l(\a)\, d\a\Bigr)^{-1}.
\end{split}
\end{equation}
From Lemma \ref{lem: abar} we know  that $\underline{\a}_n \geq \b/2$ for large enough $n$, hence by Assumption~\ref{ass: Hyperprior}, Lemma~\ref{lem: Hyperprior}, and the definition of $\underline{\a}_n$,
\[
\int_{\underline{\a}_n+1/(2\log n)}^{\underline{\a}_n+1/\log n}\l(\a)\, d\a  \geq C_1(2\log n)^{-C_2}\exp\bigl(-C_3\exp(\sqrt{\log n}/3)\bigr)
\]
for some $C_1, C_2, C_3 > 0$.
Therefore the right hand side of \eqref{eq: SecondTermPost} is bounded above by
a constant times
\[
\exp\Bigl(-\frac{K n^{1/(1+2\sqrt{\log n}+2p)}}{1+2\sqrt{\log n}+2p} \Bigr)(\log n)^{C_2}\exp\Bigl(C_3\exp\Bigl(\frac{\sqrt{\log n}}{3}\Bigr)\Bigr).
\]
It is easy to see that this quantity tends to $0$ as $n \to \infty$.

In bounding the third term on the right hand side of \eqref{eq: FullPost} we replace the supremum over
$\|\mu_0\|_\beta \le R$ by the suprema over the sets $P_n$ and  $Q_n$ defined in the beginning of Section \ref{sec: ProofSob}.
The supremum over
$Q_n$  is bounded above by
\[
 \sup_{\mu_0 \in Q_n} \E_0 \sup_{\a\in[\underline{\a}_n,\infty)} \Pi_\a(\|\mu-\mu_0\|\geq M_n
 L_nn^{-\b/(1+2\b + 2p)}|Y).
\]%
This goes to zero, as follows from Section~\ref{sec: 63} and Markov's inequality.
The supremum over $P_n$ we write as
%
{
\begin{equation}\label{eq: ThirdTermA}
\begin{split}
& \sup_{\mu_0\in P_n} \E_0 \Bigl(\int_{\underline{\a}_n}^{\overline{\a}_n}\l_n(\a|Y)\Pi_\a(\|\mu-\mu_0\|\geq M_n
L_nn^{-\b/(1+2\b + 2p)}|Y)\, d\a\\
&\qquad{+}\int_{\overline{\a}_n}^\infty\l_n(\a|Y)\Pi_\a(\|\mu-\mu_0\|\geq M_n
L_nn^{-\b/(1+2\b + 2p)}|Y)\, d\a\Bigr)1_{A_n}.
\end{split}
\end{equation}
}
The first term in \eqref{eq: ThirdTermA} is bounded above by
{
\[
\begin{split}
 \sup_{\mu_0 \in P_n}  \E_0\sup_{\a\in[\underline{\a}_n,\overline{\a}_n]} \Pi_\a(\|\mu-\mu_0\|\geq M_n
 L_nn^{-\b/(1+2\b + 2p)}|Y).
\end{split}
\]
}
This goes to zero, following from Sections \ref{sec: 61} and \ref{sec: brisk} and Markov's inequality. In Section~\ref{sec: OverA} we have shown that the differentiated log-likelihood function $\MM_n$ on the interval $[\overline{\a}_n, \infty)$ can increase maximally by
a multiple of
\[
\frac{n^{1/(1+2\overline{\a}_n+2p)}(\log n)^2}{1+2\overline{\a}_n+2p}.
\]
Moreover, in Section~\ref{sec: IntervalA} we have shown that for $\a \in [\overline{\a}_n-1/\log n,\overline{\a}_n]$,
\[
\ell'_n(\a) = \MM_n(\a) < -M\frac{n^{1/(1+2\overline{\a}_n+2p)}(\log n)^3}{1+2\overline{\a}_n+2p}
\]
on the event $A_n$, and $M$ can be made arbitrarily large by increasing the constant  $L$ in the definition of $\overline{\a}_n$. Therefore the integral of $\MM_n(\a)$ on $[\overline{\a}_n -  1/\log n, \overline{\a}_n -1/(2\log n)]$ is bounded above by
\[
 -\frac{M}{2}\frac{n^{1/(1+2\overline{\a}_n+2p)}(\log n)^2}{1+2\overline{\a}_n+2p},
\]
and by choosing a large enough constant $L$ in the definition of $\overline{\a}_n$ it holds that for some $N > 0$,
\[
\ell_n(\a) \leq \ell_n\Bigl(\overline{\a}_n - \frac{1}{2\log n}\Bigr) - N\frac{n^{1/(1+2\overline{\a}_n+2p)}(\log n)^2}{1+2\overline{\a}_n+2p}
\]
for $\a \in [\overline{\a}_n, \infty)$, and
\[
\ell_n(\a) \geq \ell_n\Bigl(\overline{\a}_n - \frac{1}{2\log n}\Bigr)
\]
for $\a \in [\overline{\a}_n - 1/\log n, \overline{\a}_n - 1/(2\log n)]$. These bounds lead to the following bounds for the likelihood:
\[
L_n(\a) \leq L_n\Bigl(\overline{\a}_n - \frac{1}{2\log n}\Bigr)\exp\Bigl(- N\frac{n^{1/(1+2\overline{\a}_n+2p)}(\log n)^2}{1+2\overline{\a}_n+2p}\Bigr)
\]
for $\a \in [\overline{\a}_n, \infty)$, and
\[
L_n(\a) \geq L_n\Bigl(\overline{\a}_n - \frac{1}{2\log n}\Bigr)
\]
for $\a \in [\overline{\a}_n - 1/\log n, \overline{\a}_n - 1/(2\log n)]$. Similarly to the upper bound for the second term of \eqref{eq: FullPost} we now write
\[
\begin{split}
&\sup_{\mu_0\in P_n} \E_0\int_{\overline{\a}_n}^\infty \l_n(\a|Y)\, d\a\leq \sup_{\mu_0\in P_n} \E_0\frac{\int_{\overline{\a}_n}^\infty \l(\a)L_n(\a)\, d\a}{\int_0^\infty \l(\a)L_n(\a)\, d\a} \\
&\qquad \leq \sup_{\mu_0\in P_n}  \exp\Bigl(- N\frac{n^{1/(1+2\overline{\a}_n+2p)}(\log n)^2}{1+2\overline{\a}_n+2p}\Bigr)\frac{\int_{\overline{\a}_n}^\infty \l(\a)\, d\a}{\int_{\overline{\a}_n-1/\log n}^{\overline{\a}_n-1/(2\log n)} \l(\a)\, d\a}.
\end{split}
\]
Since $\overline{\a}_n \geq \underline{\a}_n \geq \b/2$ for $n$ large enough, Assumption~\ref{ass: Hyperprior} and Lemma~\ref{lem: Hyperprior} imply that
\[
\frac{\int_{\overline{\a}_n}^\infty \l(\a)\, d\a}{\int_{\overline{\a}_n-1/\log n}^{\overline{\a}_n-1/(2\log n)} \l(\a)\, d\a} \leq C_4(\log n)^{C_5}\exp\bigl(C_6\overline{\a}_n^{C_7}\bigr).
\]
Since $\overline{\a}_n \leq \log n/(2\log 2)-1/2-p$ {for $\mu_0\in P_n$}, the right-hand side of the preceding display is bounded above by
\[
C_4\exp\bigl(-2C_{9}(\log 2)(\log n)\bigr)(\log n)^{C_5}\exp\Bigl(C_6\Bigl(\frac{\log n}{2\log 2}-\frac{1}{2}-p\Bigr)^{C_7}\Bigr),
\]
which tends to zero for any fixed constant $C_7$ smaller than $1$.

\begin{lemma}\label{lem: Hyperprior}

{

Suppose that for $c_1 > 0$, $c_2 \geq 0$, $c_3 \in \RR$, with $c_3>1$ if $c_2=0$,  and $c_4>0$, the prior  density $\l$ satisfies
\[
c_4^{-1}\a^{-c_3} \exp(-c_2\a) \leq \l(\a) \leq c_4\a^{-c_3} \exp(-c_2\a)
\]
for $\a \geq c_1$.
}
Then there exist positive constants $C_1, \ldots, C_6$ and $C_7 < 1$ depending on $c_1$ only such that for all $x \geq c_1$, every $\d_n \to 0$, and $n$ large enough
\[
\int_{x+\d_n}^{x+2\d_n} \l(\a)\, d\a \geq C_1\d_n^{C_2}\exp\Bigl(-C_3\exp\Bigl(\frac{x}{3}\Bigr)\Bigr)
\]
and
\[
\frac{\int_x^\infty \l(\a)\, d\a}{\int_{x-2\d_n}^{x-\d_n}\l(\a)\, d\a} \leq C_4\d_n^{-C_5}\exp(C_6x^{C_7}).
\]
\end{lemma}

\begin{proof}
The proof only involves straightforward calculus.
%
%
%
%
\end{proof}

\section{Auxiliary lemmas}\label{sec: Appendix}

In this section we collect several lemmas that we use throughout the proofs to
upper and lower bound certain sums.

\begin{lemma}\label{lem: LogSeries}
Let $c > 0$ and $r \geq 1+c$.
\begin{itemize}
	\item[(i)] For $n \geq 1$
\[
\sum_{i=1}^\infty \frac{n \log i}{i^{r}+n} \leq \Bigl(2+\frac{2}{c}+\frac{2}{c^2\log 2}\Bigr)\frac{n^{1/r}\log n}{r}.
\]
	\item[(ii)] If $r > (\log n)/(\log 2)$, then for $n \geq 1$
\[
\sum_{i=1}^\infty \frac{n \log i}{i^{r}+n} \leq \Bigl(1 +\frac{2}{c} + \frac{2}{c^2\log 2}\Bigr)(\log 2)n2^{-r}.
\]
\end{itemize}
\end{lemma}

\begin{proof}
First consider $r\leq (\log n)/(\log 2)$, which implies that $n^{1/r} \geq 2$. We split the series in two parts, and bound the denominator $i^{r}+n$ by $n$ or $i^{r}$. Since $\log i$ is increasing, we see that
\[
\sum_{i=1}^{\lfloor n^{1/r}\rfloor}\log i \leq  \frac{n^{1/r}\log n}{r}.
\]
Since $f(x) = x^{-\g}\log x$ is decreasing for $x \geq e^{1/\g}$, we see that $i^{-r}\log i$ is decreasing on interval $\bigl[\lceil n^{1/r}\rceil, \infty\bigr)$ for $n \geq e$. Therefore
\[
\sum_{i=\lceil n^{1/r}\rceil}^{\infty}\frac{n\log i}{i^{r}} \leq n \frac{\log \lceil n^{1/r}\rceil}{\lceil n^{1/r}\rceil^{r}} + n\int_{\lceil n^{1/r}\rceil}^\infty \frac{\log x}{x^{r}}\, dx.
\]
Since $\lceil x\rceil/x \leq 2$ for $x \geq 1$, and $n^{1/r} \geq 2$,
\[
n \frac{\log \lceil n^{1/r}\rceil}{\lceil n^{1/r}\rceil^{r}} \leq 2 \log n^{1/r} \leq \frac{n^{1/r}\log n}{r}.
\]
Moreover
\[
\int_{\lceil n^{1/r}\rceil}^\infty \frac{\log x}{x^{r}}\, dx \leq \int_{n^{1/r}}^\infty \frac{\log x}{x^r}\, dx= n^{-1+1/r}\frac{(r-1)\log n^{1/r}+1}{(r-1)^2}.
\]
Since $r \geq 1+c$, we have
\[
\frac{\log n^{1/r}}{r-1} \leq \frac{1}{c}\cdot \frac{\log n}{r},
\qquad
\frac{1}{(r-1)^2}  \leq \frac{\log n^{1/r}}{(r-1)^2\log 2} \leq \frac{1}{c^2\log 2}\cdot \frac{\log n}{r}.
\]
This proves (i) for the case $r\leq (\log n)/(\log 2)$.

We now consider  $r > (\log n)/(\log 2)$, which implies that $n^{1/r} < 2$. We have
\[
\sum_{i=2}^\infty \frac{n\log i}{i^{r}+n}\leq n\sum_{i=2}^\infty \frac{\log i}{i^{r}} \leq n2^{-r}\log 2 + n\int_2^\infty x^{-r}\log x \, dx,
\]
by monotonicity of the function $f$ defined above (with $\g = r$). We have
\[
\int_2^\infty x^{-r}\log x \, dx = 2^{1-r}\frac{(r-1)\log 2 + 1}{(r-1)^2},
\]
and since $r \geq 1+c$
\[
\frac{\log 2}{r-1}\leq \frac{\log 2}{c}, \qquad\qquad\qquad \frac{1}{(r-1)^2}\leq \frac{1}{c^2},
\]
which finishes the proof of (ii).

To complete the proof of (i),  we consider the function $f(x) = 2^{-x}x$ and note that it is decreasing for $x > 1/\log 2$. Therefore $n2^{-r} = (n2^{-r}r)/r \leq (\log n)/(r\log 2)$, for $n \geq 3$. Since $1 \leq n^{1/r}$, we get the desired result.
\end{proof}

\begin{lemma}\label{lem: LogM}
For any $m>0$, $l \geq 1$, $r_0>0$, $r \in (0, r_0]$, $s \in (0, rl-2]$, and $n \geq e^{2mr_0}$
\[
\sum_{i=1}^\infty \frac{i^{s}(\log i)^m}{(i^{r}+n)^l} \leq 4n^{(1+s-lr)/r}\frac{(\log n)^m}{r^m}.
\]
The same upper bound holds for $m = 0$, $r \in (0, \infty)$, $s \in (0, rl-1)$, and $n \geq 1$.
\end{lemma}

\begin{proof}
We deal with this sum by splitting the sum in the parts $i \leq n^{1/r}$ and $i > n^{1/r}$. In the first range we bound the sum by
\[
\sum_{i=1}^{n^{1/r}}n^{-l}i^{s}(\log i)^m \leq n^{1/r}n^{-l+s/r}\frac{(\log n)^m}{r^m},
\]
by monotonicity of the function $f(x)= x^{s}(\log x)^m$.

Suppose that $m > 0$. The derivative of the function $f(x) = x^{-1/2}(\log x)^m$ is $f'(x) = x^{-3/2}(\log x)^{m-1}(m-(\log x)/2)$, hence it is monotone decreasing for $x \geq e^{2m}$. Since $n^{1/r}\geq n^{1/r_0}$ and $n>e^{2mr_0}$, the function $f$ is decreasing on interval $[n^{1/r}, \infty)$. Therefore we bound the sum over the second range by
\[
\sum_{i=n^{1/r}}^\infty i^{s-rl}(\log i)^m \leq n^{-1/(2r)}\frac{(\log n)^m}{r^m}\sum_{i=n^{1/r}}^\infty i^{1/2+s-rl}.
\]
Since $s \leq rl-2$, $i^{1/2+s-rl}$ is decreasing and $rl-s-3/2 \geq 1/2$. We get
\begin{align*}
\sum_{i=n^{1/r}}^\infty i^{1/2+s-rl}
&\leq n^{(1/2+s-rl)/r} + \int_{n^{1/r}}^\infty x^{1/2+s-rl}\, dx\\
&  =n^{(1/2+s-rl)/r} + \frac{1}{-3/2-s+rl}n^{(3/2+s-rl)/r}\\
&\leq 3n^{(3/2+s-rl)/r}.
\end{align*}
In the case $m=0$, we use monotonicity of $i^{s-rl}$ for all $i \geq 1$.
\end{proof}

\begin{lemma}\label{lem: LogLower}
For any $r \in (1, (\log n)/(2\log(3e/2))]$, and $\g > 0$,
\[
\sum_{i=1}^{\infty}\frac{n^\g\log i}{(i^{r}+n)^\g} \geq \frac{1}{3\cdot 2^{\g}r}n^{1/r}\log n.
\]
\end{lemma}

\begin{proof}
In the range $i\leq n^{1/r}$ we have $i^r + n \leq 2n$, thus
\[
\sum_{i=1}^{\infty}\frac{n^\g\log i}{(i^{r}+n)^\g} \geq \frac{1}{2^\g}\sum_{i=1}^{\lfloor n^{1/r}\rfloor}\log i \geq \frac{1}{2^\g}\int_1^{\lfloor n^{1/r}\rfloor}\log x\, dx \geq \frac{1}{2^\g}\int_1^{(2/3) n^{1/r}}\log x\, dx,
\]
since $n^{1/r}\geq 2$ and $\lfloor x\rfloor \geq 2x/3$ for $x \geq 2$. The latter integral equals $(2/3)n^{1/r}\bigl(\log ((2/3)n^{1/r}) - 1\bigr) + 1.$
Since $\log n \geq 2\log(3e/2)r$ implies that $(\log n)/(2r) \geq \log(3e/2)$, we have
\[
\frac{2}{3}n^{1/r}\Bigl(\log\Bigl(\frac{2}{3}n^{1/r}\Bigr) - 1\Bigr) = \frac{2}{3}n^{1/r}\Bigl(\frac{1}{r}\log n - \log\frac{3e}{2}\Bigr)\geq \frac{1}{3r}n^{1/r}\log n.
\]
\end{proof}

\begin{lemma}\label{lem: Botond'sTrick}
Let $m$, $i$, $r$, and $\xi$ be positive reals. Then for $n \geq e^m$
\[
\frac{ni^{r}\bigl(r\log i\bigr)^m}{(i^{r}+n)^2}\leq (\log n)^m, \qquad \text{and} \qquad \frac{n^\xi\bigl(r\log i\bigr)^{\xi m}}{(i^{r}+n)^\xi}\leq (\log n)^{\xi m}.
\]
\end{lemma}

\begin{proof}
Assume first that $i\leq n^{1/r}$, then the left hand side of the first inequality is bounded above by
\[
\frac{n^2\bigl(r\log n^{1/r}\bigr)^m}{n^2}=(\log n)^m.
\]
Next assume that $i> n^{1/r}$. The derivative of the function $f(x)=x^{-c}(\log x)^m$ is $f'(x)=x^{-c-1}(\log x)^{m-1}\big(-c(\log x)+m\big)$, hence $f(x)$ is monotone decreasing for $x\geq e^{m/c}$. Therefore the function $i^{-r}(\log i)^m$ is monotone decreasing for $i\geq e^{m/r}$ and since by assumption $i> n^{1/r}$, we get that for $n\geq e^m$ the function $f(i)=i^{-r}(\log i)^m$ takes its maximum at $i=n^{1/r}$. Hence the left hand side of the inequality is bounded above by
\begin{align*}
n\bigl(r\log i\bigr)^mi^{-r}\leq n r^m \bigl(\log n^{1/r}\bigr)^mn^{-1}= (\log n)^m.
\end{align*}
The second inequality can be proven similarly.
\end{proof}

%

\bibliographystyle{acm}
\bibliography{inverseadaptive_revision}

\def\cprime{$'$}
\begin{thebibliography}{10}

\bibitem{Griek}
{\sc Agapiou, S., Larsson, S., and Stuart, A.~M.}
\newblock Posterior contraction rates for the {B}ayesian approach to linear
  ill-posed inverse problems.
\newblock {\em Stochastic Process. Appl. (to appear)\/} (2013).

\bibitem{Bailer}
{\sc Bailer-Jones, C. A.~L.}
\newblock Bayesian inference of stellar parameters and interstellar extinction
  using parallaxes and multiband photometry.
\newblock {\em Mon. Not. R. Astron. Soc. 411}, 1 (2011), 435--452.

\bibitem{BelitserEB}
{\sc Belitser, E., and Enikeeva, F.}
\newblock Empirical {B}ayesian test of the smoothness.
\newblock {\em Math. Methods Statist. 17}, 1 (2008), 1--18.

\bibitem{Belitser}
{\sc Belitser, E., and Ghosal, S.}
\newblock Adaptive {B}ayesian inference on the mean of an infinite-dimensional
  normal distribution.
\newblock {\em Ann. Statist. 31}, 2 (2003), 536--559.

\bibitem{Cai}
{\sc Cai, T.~T.}
\newblock On adaptive wavelet estimation of a derivative and other related
  linear inverse problems.
\newblock {\em J. Statist. Plann. Inference 108}, 1-2 (2002), 329--349.

\bibitem{Ismael}
{\sc Castillo, I.}
\newblock Lower bounds for posterior rates with {G}aussian process priors.
\newblock {\em Electron. J. Stat. 2\/} (2008), 1281--1299.

\bibitem{Cavalier}
{\sc Cavalier, L.}
\newblock Nonparametric statistical inverse problems.
\newblock {\em Inverse Problems 24}, 3 (2008), 034004, 19.

\bibitem{Cavalier2}
{\sc Cavalier, L.}
\newblock Inverse {P}roblems in {S}tatistics.
\newblock In {\em Inverse {P}roblems and {H}igh-{D}imensional {E}stimation:
  {S}tats in the {C}h\^{a}teau {S}ummer {S}chool}, vol.~203 of {\em Lecture
  {N}otes in {S}tatistics.} Springer, 2011, pp.~3--96.

\bibitem{CavURE}
{\sc Cavalier, L., Golubev, G.~K., Picard, D., and Tsybakov, A.~B.}
\newblock Oracle inequalities for inverse problems.
\newblock {\em Ann. Statist. 30}, 3 (2002), 843--874.

\bibitem{CavRH}
{\sc Cavalier, L., and Golubev, Y.}
\newblock Risk hull method and regularization by projections of ill-posed
  inverse problems.
\newblock {\em Ann. Statist. 34}, 4 (2006), 1653--1677.

\bibitem{CavTsy}
{\sc Cavalier, L., and Tsybakov, A.}
\newblock Sharp adaptation for inverse problems with random noise.
\newblock {\em Probab. Theory Related Fields 123}, 3 (2002), 323--354.

\bibitem{JZ}
{\sc de~Jonge, R., and van Zanten, J.~H.}
\newblock Adaptive nonparametric {B}ayesian inference using location-scale
  mixture priors.
\newblock {\em Ann. Statist. 38}, 6 (2010), 3300--3320.

\bibitem{Simoni}
{\sc Florens, J., and Simoni, A.}
\newblock Regularized posteriors in linear ill-posed inverse problems.
\newblock {\em Scand. J. Statist. 39}, 2 (2012), 214--235.

\bibitem{Gao}
{\sc Gao, P., H. A. R.~M., and Lawrence, N.~D.}
\newblock Gaussian process modelling of latent chemical species: applications
  to inferring transcription factor activities.
\newblock {\em Bioinformatics 24\/} (2008), 70--75.

\bibitem{GGvdV}
{\sc Ghosal, S., Ghosh, J.~K., and van~der Vaart, A.~W.}
\newblock Convergence rates of posterior distributions.
\newblock {\em Ann. Statist. 28}, 2 (2000), 500--531.

\bibitem{GVnoniid}
{\sc Ghosal, S., and van~der Vaart, A.}
\newblock Convergence rates of posterior distributions for non-i.i.d.
  observations.
\newblock {\em Ann. Statist. 35}, 1 (2007), 192--223.

\bibitem{JohnstoneSilverman}
{\sc Johnstone, I.~M., and Silverman, B.~W.}
\newblock Needles and straw in haystacks: empirical {B}ayes estimates of
  possibly sparse sequences.
\newblock {\em Ann. Statist. 32}, 4 (2004), 1594--1649.

\bibitem{JS2005}
{\sc Johnstone, I.~M., and Silverman, B.~W.}
\newblock Empirical {B}ayes selection of wavelet thresholds.
\newblock {\em Ann. Statist. 33}, 4 (2005), 1700--1752.

\bibitem{KVZ}
{\sc Knapik, B.~T., van~der Vaart, A.~W., and van Zanten, J.~H.}
\newblock Bayesian inverse problems with {G}aussian priors.
\newblock {\em Ann. Statist. 39}, 5 (2011), 2626--2657.

\bibitem{KVZHeat}
{\sc Knapik, B.~T., van~der Vaart, A.~W., and van Zanten, J.~H.}
\newblock Bayesian recovery of the initial condition for the heat equation.
\newblock {\em Comm. Statist. Theory Methods 42}, 7 (2013), 1294--1313.

\bibitem{Lashkari}
{\sc Lashkari, D., Sridharan, R., Vul, E., Hsieh, P.-J., Kanwisher, N., and P.,
  G.}
\newblock Search for patterns of functional specificity in the brain: A
  nonparametric hierarchical {B}ayesian model for group f{MRI} data.
\newblock {\em NeuroImage 59}, 2 (2012), 1348--1368.
\newblock To appear.

\bibitem{MarteauRH}
{\sc Marteau, C.}
\newblock Risk hull method for spectral regularization in linear statistical
  inverse problems.
\newblock {\em ESAIM Probab. Stat. 14\/} (2010), 409--434.

\bibitem{MarteauHull}
{\sc Marteau, C.}
\newblock The {S}tein hull.
\newblock {\em J. Nonparametr. Stat. 22}, 5-6 (2010), 685--702.

\bibitem{Oh}
{\sc Oh, S.-H., and Kwon, B.-D.}
\newblock Geostatistical approach to {B}ayesian inversion of geophysical data:
  {M}arkov chain {M}onte {C}arlo method.
\newblock {\em Earth Planets Space 53\/} (2001), 777--791.

\bibitem{Orbanz}
{\sc Orbanz, P., and Buhmann, J.~M.}
\newblock Nonparametric {B}ayesian image segmentation.
\newblock {\em Int. J. Comput. Vis. 77\/} (2008), 25--45.

\bibitem{JSC}
{\sc Petrone, S., Rousseau, J., and Scricciolo, C.}
\newblock Bayes and empirical {B}ayes: do they merge?
\newblock Preprint.

\bibitem{Ray}
{\sc {Ray}, K.}
\newblock {Bayesian inverse problems with non-conjugate priors}.
\newblock {\em ArXiv e-print 1209.6156\/} (2012).

\bibitem{Rivoirard}
{\sc Rivoirard, V., and Rousseau, J.}
\newblock Bernstein--von {M}ises theorem for linear functionals of the density.
\newblock {\em Ann. Statist. 40}, 3 (2012), 1489--1523.

\bibitem{SG}
{\sc Shen, W., and Ghosal, S.}
\newblock {MCMC}-free adaptive {B}ayesian procedures usingrandom series prior.
\newblock Preprint.

\bibitem{Shen}
{\sc Shen, X., and Wasserman, L.}
\newblock Rates of convergence of posterior distributions.
\newblock {\em Ann. Statist. 29}, 3 (2001), 687--714.

\bibitem{Stuart}
{\sc Stuart, A.~M.}
\newblock Inverse problems: a {B}ayesian perspective.
\newblock {\em Acta Numer. 19\/} (2010), 451--559.

\bibitem{SVZ}
{\sc Szab\'o, B.~T., van~der Vaart, A.~W., and van Zanten, J.~H.}
\newblock Frequentist coverage and adaptation of nonparametric {B}ayesian
  credible sets.
\newblock {\em In preparation\/} (2013).

\bibitem{Tier}
{\sc Tierney, L.}
\newblock Markov chains for exploring posterior distributions.
\newblock {\em Ann. Statist. 22}, 4 (1994), 1701--1762.

\bibitem{vdVvZAdaptive}
{\sc van~der Vaart, A.~W., and van Zanten, J.~H.}
\newblock Adaptive {B}ayesian estimation using a {G}aussian random field with
  inverse gamma bandwidth.
\newblock {\em Ann. Statist. 37}, 5B (2009), 2655--2675.

\bibitem{vdVW}
{\sc van~der Vaart, A.~W., and Wellner, J.~A.}
\newblock {\em Weak convergence and empirical processes}.
\newblock Springer-Verlag, New York, 1996.

\bibitem{Zhang}
{\sc Zhang, C.-H.}
\newblock General empirical {B}ayes wavelet methods and exactly adaptive
  minimax estimation.
\newblock {\em Ann. Statist. 33}, 1 (2005), 54--100.

\bibitem{Zhao}
{\sc Zhao, L.~H.}
\newblock Bayesian aspects of some nonparametric problems.
\newblock {\em Ann. Statist. 28}, 2 (2000), 532--552.

\end{thebibliography}

\end{document}